\documentclass{amsart}
\usepackage{macros}
\upsymbols{Tor,Isom,Div,Hom,Br,inv,WC,nr,red,GL,SL,Aut,tors,Hom,Inv}
\upsymbols{Gal,res,im,coker,Frob,rk,Res,Ext,disc,Id,FS,cyc,PGL,cor,Vol}
\bbsymbols{A,G,Z,Q,R,F,C,E,N}
\scrsymbols{F,C,S,L,M,G}
\calsymbols{O,T,F,G,C,U,D,N}
\fraksymbols{p,P}
\boldsymbols{x,B,z,k,a,c}

\newcommand{\sel}[1]{\textup{Sel}_{#1}}

\makeatletter
\newcommand{\addresseshere}{%
  \enddoc@text\let\enddoc@text\relax
}
\makeatother

\title{2-Selmer Groups over Multiquadratic Extensions}
\author{Ross Paterson}
\subjclass[2010]{Primary 11G05; Secondary 11G07, 11N36, 11N45, 11R45}
\makeatletter
\hypersetup{
pdftitle={\@title},
pdfauthor={Ross Paterson},
pdfsubject={Number Theory},
pdfkeywords={}
}
\makeatother
\address{School of Mathematics, University of Bristol, Bristol, BS8 1TW, UK, and the Heilbronn Institute for Mathematical Research, Bristol, UK.}
\email{rosspatersonmath@gmail.com}
\urladdr{\url{https://ramifiedprime.github.io}}

\begin{document}
\begin{abstract}
    Let $K/\QQ$ be a multiquadratic extension.  We investigate the average dimension of $2$-Selmer groups over $K$ for the family of all elliptic curves over $\QQ$ (ordered by height).  We give upper and lower bounds for this average.  In the special case of quadratic fields, these bounds are arbitrarily close for a positive proportion of $K/\QQ$.  Our bounds are achieved by studying the genus theory invariant for $2$-Selmer groups over such fields, whose average we similarly bound and, in many cases, determine.  We make use of a variant of the Ekedahl sieve for local sums, which we present in appropriate generality for further applications.
\end{abstract}
\maketitle
\section{Introduction}

The purpose of this article is to investigate the average dimension of $2$-Selmer groups of elliptic curves over multiquadratic extensions.  For each number field $K/\QQ$, let us define
\begin{equation*}
A(K;X):=\frac{\sum\limits_{E\in\Epsilon(X)}\dim_{\FF_2}\sel{2}(E/K)}{\#\Epsilon(X)},
\end{equation*}
where $\Epsilon(X)$ is the set of all elliptic curves defined over $\QQ$ with (na\"ive) height at most $X$ (see \Cref{subsec:notation and conventions} for a formal definition).  We shall be interested in this quantity as $X\to\infty$, and so we write $A^+(K):=\limsup_{X\to\infty}A(K;X)$ and similarly $A^-(K):=\liminf_{X\to\infty}A(K;X)$.

For $K=\QQ$, breakthrough work of Bhargava--Shankar \cite{MR3272925}*{Corollary 1.2} showed that $A^+(\QQ)\leq 1.5$, which in turn provided the first bound on the average rank of elliptic curves over $\QQ$.  Further work of the aforementioned authors \cite{bhargava3Sel}*{Theorem 5} then showed, assuming a standard conjecture, that $A^-(\QQ)>0$.  More recently, a lower bound of $A^-(\QQ)\geq0.4$ can be seen unconditionally using the second moment estimate of Bhargava--Shankar--Swaminathan \cite{swaminathan2021second}*{Theorem 1.1}.  Additionally, previous work of the author \cite{Paterson2021}*{Theorem 1.6} gives explicit upper bounds for $A^+(K)$ when $K/\QQ$ is a Galois extension of degree a power of $2$.

In this article, we focus on the situation where $K/\QQ$ is a multiquadratic extension, i.e. a nontrivial finite Galois extension whose Galois group is an elementary abelian $2$-group.  In this setting we prove stronger upper bounds for $A^+(K)$, as well as giving positive lower bounds for $A^-(K)$ respectively.  In the case of quadratic fields, these bounds are often very close.

\subsection{Main Results}
In order to state our results, we must introduce some notation.
\begin{definition}\label{def:GvK Invariants}
For each multiquadratic field $K/\QQ$, and each place $v\in\places_\QQ$, we define a local invariant $\cG_v(K)$ as follows.  For primes $p\geq 5$, we set
\[\cG_{p}(K)=\begin{cases}
    \frac{p(p^5+1)(p-1)(5p^3+2p^2+3)}{6(p^{10}-1)}&\substack{\textnormal{if }K/\QQ\textnormal{ is ramified and}\\\textnormal{quadratic at }p,}\\
    \frac{p(p-1)(3p^7+3p^6+3p^5+p^4+2p^3+p^2+3p+3)}{3(p+1)(p^{10}-1)}&\substack{\textnormal{if }K/\QQ\textnormal{ is unramified and}\\\textnormal{quadratic at }p,}\\
    \frac{p(p^5+1)(10p^5 + 4p^4 - 7p^3 + 5p^2 - 12)}{12(p^{10}-1)(p+1)}&\substack{\textnormal{if }K/\QQ\textnormal{ is biquadratic at }p,}\\
    0&\substack{\textnormal{if }K/\QQ\textnormal{ is totally split at }p.}
\end{cases}\]
For the remaining finite places we define
\begin{align*}
    \cG_2(K)&=\begin{cases}
        0&\substack{\textnormal{if }K/\QQ\textnormal{ is totally split at }2,}\\
        \frac{3}{1-2^{-10}}=\frac{1024}{341}&\substack{\textnormal{else;}}
    \end{cases}
    &\cG_3(K)&=\begin{cases}
        0&\substack{\textnormal{if }K/\QQ\textnormal{ is totally split at }3,}\\
        \frac{2}{1-3^{-10}}=\frac{59049}{29524}&\substack{\textnormal{else.}}
    \end{cases}
\end{align*}
For the archimedean place we define
\[\cG_{\infty}(K)=\begin{cases}
    \frac{1}{10}&\substack{\textnormal{if }K\textnormal{ is imaginary,}}\\
    0&\substack{\textnormal{else.}}
\end{cases}\]
\end{definition}
\noindent Our main result for $2$-Selmer groups is then summarised by the following.
\begin{theorem}[\Cref{cor:MQ Selmer Avg}]\label{thm:INTRO MQ Selmer Avg}
    Let $K/\QQ$ be a multiquadratic extension.  Then 
    \[\sum_{\substack{v\in\places_\QQ\\v\nmid 6}}\cG_v(K)\leq A^-(K)\leq A^+(K)\leq \braces{[K:\QQ]-1}\sum_{v\in\places_\QQ}\cG_v(K) + \bigo{\braces{\frac{23}{24}}^{\omega\braces{\Delta_K}}}.\]
\end{theorem}
\begin{rem}
    Examining the local terms $\cG_p(K)$, the above theorem shows that if $K/\QQ$ is a quadratic field then
    \[\omega(\Delta_K)\ll A^-(K)\leq A^+(K)\ll \omega(\Delta_K),\]
    where $\omega(D)$ denotes the number of distinct prime factors of $D$.  One should think of this as a statistical Selmer group analogue of Gauss' genus theory for class groups: that $\dim\Cl_K[2]\approx\omega(\Delta_K)$ for quadratic fields $K/\QQ$.
\end{rem}
    In fact, the error term above is an explicit local product arising from the work in \cite{PatQTRV}.  See \Cref{thm:MQ Selmer Avg} for the explicit version of this.  In particular, the upper bound here improves that in \cite{Paterson2021}*{Theorem 1.6}.
\begin{rem}
    Our result shows an interesting departure from the picture for $\sel{2}(E/\QQ)$.  In that setting, statistical results for most quadratic twist families tend to agree with those for the family of all elliptic curves, at least where both are known.  However, at least for $K/\QQ$ quadratic, for $\sel{2}(E/K)$ the average dimension in quadratic twist families generally goes to infinity (see \cite{MR4400944}*{Corollary 5.11}) whereas we show in \Cref{thm:INTRO MQ Selmer Avg} that it is finite in the family of all elliptic curves.
\end{rem}

\subsection{The Genus Theory Invariant}
The results above arise from a detailed study of a quantity known as the genus theory invariant associated to each $\sel{2}(E/K)$.  This is a sum of local invariants, whose definition we recall below.
\begin{definition}[\cite{Paterson2021}*{Definition 4.10}]\label{def:GTInvt}
    The genus theory invariant of the $2$-Selmer group of an elliptic curve $E/\QQ$ arising from a finite Galois extension $K/\QQ$ is
    \[g_2(K/\QQ;E):=\sum_{v\in\places_\QQ} \dim_{\FF_2} E(\QQ_v)/\left(N_{K_w/\QQ_v}E(K_w)+2E(\QQ_v)\right),\]
    where, in each summand, $w\in \Omega_K$ is any place of $K$ lying over $v$.
\end{definition}
\begin{rem}
    By \cite{mazur1972rational}*{Corollary 4.4} the norm map is surjective at unramified places of good reduction, so for fixed $E,K$ the genus theory is in fact a finite sum of nonzero terms.
\end{rem}

In \Cref{sec:main results} we explain, in the setting where $K/\QQ$ is multiquadratic, how to bound $\dim_{\FF_2}\sel{2}(E/K)$ in terms of the genus theory invariant and an auxiliary Selmer group known as the corestriction Selmer group.  The average behaviour of this auxiliary Selmer group is the subject of \cite{PatQTRV}, and leads to the error term in \Cref{thm:INTRO MQ Selmer Avg}.  The bulk of this article is then concerned with determining the average of $g_2(K/\QQ;E)$ as $E/\QQ$ varies in the family of all elliptic curves, which in turn allows us to conclude the bounds in \Cref{thm:INTRO MQ Selmer Avg}.  Our main theorem for the genus theory is then as follows.

\begin{theorem}[\Cref{thm:MQGenusTheoryAverage}]\label{thm:INTRO MQGenusTheoryAverage}
Let $K/\QQ$ be a multiquadratic extension, and take notation as in \Cref{def:GvK Invariants}.  Let
\[\cG(K;X):=\frac{\sum\limits_{(A,B)\in\Epsilon(X)}g_2(K/\QQ;E_{A,B})}{\#\Epsilon(X)},\]
and write $\cG^+(K):=\limsup_{X\to\infty}\cG(K;X)$, and $\cG^-(K):=\liminf_{X\to\infty}\cG(K;X)$. Then the following bounds hold.
\[\sum_{\substack{v\in\places_\QQ\\v\nmid 6}}\cG_v(K)\leq \cG^-(K)\leq \cG^+(K)\leq \sum_{v\in\places_\QQ}\cG_v(K).\]
\end{theorem}
We note as a corollary that these bounds actually agree in certain cases, for example when $K/\QQ$ is quadratic with discriminant congruent to $1$ mod $24$.
\begin{corollary}
    Let $K/\QQ$ be a multiquadratic extension in which $2$ and $3$ are totally split, then with notation as above 
    \[\lim_{X\to\infty}\cG(K;X)=\sum_{v\in\places_\QQ}\cG_v(K).\]
\end{corollary}

A main ingredient in the proof of this theorem is a systematic computation of local norm indices at places $v\nmid 6$, which is performed in \Cref{sec:norm indices}.

\subsection{Local sums over Ekedahl-type sets}\label{subsec:INTRO ekedahl}
We determine \Cref{thm:INTRO MQGenusTheoryAverage} using a modification of the Ekedahl sieve from \cite{MR1104780}, as developed by Poonen--Stoll \cite{MR1740984}, Bhargava \cite{bhargava2014geometric}, Bhargava--Shankar--Wang \cite{bhargava2021squarefree}, Cremona--Sadek \cite{cremona2021local}, and others.  Given $\bk=(k_1,\dots,k_n)\in\ZZ_{>0}^n$, $\bc=(c_1,\dots,c_n)\in\RR_{>0}^n$, and $X>0$, the Ekedahl sieve is concerned with the lattice points in
\[V_{\bk,\bx}(X):=\set{\bx\in\RR^n~:~\abs{x_i}^{k_i}\leq c_iX}.\]
Indeed, given a collection $\cU=(\cU_p)_{p\textnormal{ prime}}$ of subsets $\cU_p\subseteq \ZZ_p^n$ satisfying some mild technical hypotheses, the Ekedahl sieve establishes the density of the sieved set
\[V_{\bk,\bc}^\cU(X):=\set{\bx\in V_{\bk,\bc}(X)~:~\forall p,\ \bx\not\in \cU_p}\subseteq V_{\bk,\bc}(X).\]
We refer to triples $(\cU,\bk,\bc)$ which satisfy the relevant hypotheses as Ekedahl-admissible (see \Cref{def:EkedahlAdmissible}).  

In \Cref{sec:EkedahlAverages}, we consider collections $\phi=(\phi_p)_p$, where each $\phi_p:\ZZ_p^n\to \RR_{>0}$ is a locally constant function.  Given an Ekedahl-admissible triple $(\cU,\bk,\bc)$, we then consider the average value of $\sum_{p}\phi_p(\bx)$ for $\bx\in V_{\bk,\bc}^{\cU}(X)$, assuming some mild hypotheses on $\phi$.  We refer to collections $\phi$ which satisfy the relevant hypotheses as acceptable for the triple $(\cU,\bk,\bc)$ (see \Cref{def:acceptable function}).  Our main theorem for this is then the following.

\begin{theorem}[\Cref{thm:average of acceptable function over all lattice points}]\label{thm:INTRO average of acceptable function over all lattice points}
    Let $(\cU,\bk,\bc)$ be an Ekedahl-admissible triple, and $\phi=(\phi_p)_p$ be acceptable for this triple.  Then
    \[\lim_{X\to\infty}\frac{
    \sum\limits_{\bx\in V_{\bk,\bc}^\cU(X)}\braces{\sum\limits_p\phi_p(\bx)}
    }{
    \#V^{\cU}_{\bk,\bc}(X)
    }=\sum_{p\textnormal{ prime}}\frac{\int_{\ZZ_p^n\backslash \cU_p}\phi_p(\bz)d\bz}{1-\mu_p(\cU_p)}.\]
\end{theorem}

In particular, this establishes that in the sieved set $V_{\bk,\bc}^\cU(X)$ the average is simply the sum of the corresponding $p$-adic averages.  We then apply this to (the coefficient space for) the family of all elliptic curves of bounded height, where the local functions are the local norm indices from the genus theory invariant.

We expect this general machine to be useful in more generality than the application here, and so present it for the convenience of the reader.

\subsection{Outline}
In \Cref{sec:main results} we prove \Cref{thm:INTRO MQ Selmer Avg} assuming \Cref{thm:INTRO MQGenusTheoryAverage}.  The rest of the article is then concerned with establishing the latter result.

Firstly, in \Cref{sec:EkedahlAverages}, we prove the general counting result for on averages of local sums over Ekedahl-type subsets of lattices from \Cref{thm:INTRO average of acceptable function over all lattice points}.  Later in the article, we use this to convert the average of the genus theory to a sum of $p$-adic integrals of local norm indices.  In \Cref{sec:norm indices} we perform a systematic study of local norm indices in multiquadratic extensions, which is necessary for determining these local integrals.  Finally, in \Cref{sec:avg of genus theory} we bring together the results of the preceding sections to compute the $p$-adic integrals and prove \Cref{thm:INTRO MQGenusTheoryAverage}.

\subsection{Acknowledgements}
Part of this work was supported by a Ph.D. scholarship from the Carnegie Trust for the Universities of Scotland.  We are grateful to Alex Bartel and Adam Morgan for helpful conversations and comments.  

\subsection{Notation and Conventions}\label{subsec:notation and conventions}
Throughout this article, if $L/K$ is a finite Galois field extension then we write $N_{L/K}$ for the norm map (considered as an element of the group ring over $\gal(L/K$).  Moreover, if $K$ is a number field then we write $\Omega_K$ for the set of places of $K$.  If $F$ is a nonarchimedean local field then we write $v_F,\pi_F,k_F$ for the normalised valuation on $F$, a (fixed) choice of uniformiser for $F$, and the residue field of $F$.  For $n\in\ZZ$, we write $\omega(n)$ for the number of distinct prime divisors of $n$.

Throughout the article we let
\[\Epsilon=\set{(A,B)\in\ZZ^2~:~\substack{A,B\in\ZZ,\\\gcd(A^3,B^2)\textnormal{ is }12^{th}\textnormal{-power free,}\\\textnormal{and }4A^3+27B^2\neq 0}}.\]
It is well known, see e.g. \cite{silverman2009arithmetic}*{III.1}, that for every elliptic curve $E/\QQ$ there is precisely one pair $(A,B)\in\Epsilon$ such that $E\cong E_{A,B}:y^2=x^3+Ax+B$.  The (na\"ive) height of $(A,B)\in\Epsilon$ is defined to be $H(A,B):=\max\{4\abs{A}^3,27B^2\}$, and for every positive real number $X$, we write $\Epsilon(X)$ for the finite subset of $\Epsilon$ of pairs which have height at most $X$.

\section{Reduction to Genus Theory Invariant}\label{sec:main results}
In this section we begin by recalling the definition and basic properties of the genus theory invariant for $2$-Selmer groups over extension fields, as well as some important properties of it.  We then pull forward the main underpinning result of this article, on the average of this invariant for $2$-Selmer groups over multiquadratic extensions, and use this to prove \Cref{thm:INTRO MQ Selmer Avg}.  The sections that follow this one will then be directed toward proving that central averaging result. 

\subsection*{The Genus Theory Invariant}
The genus theory invariant of \Cref{def:GTInvt} is connected to our study of $2$-Selmer groups via Poitou-Tate duality.  In order to explain this, we shall require one more definition: the corestriction Selmer group.
\begin{definition}
    Let $E/\QQ$ be an elliptic curve, and $K/\QQ$ be a multiquadratic extension. For each $v\in\places_\QQ$, choose a place $w\in\places_K$ extending it and define
    \[\sC_v(K/\QQ;E):=\frac{N_{K_w/\QQ_v}E(K_w)+2E(\QQ_v)}{2E(\QQ_v)}\subseteq H^1(\QQ_v, E[2]),\]
    where the inclusion is induced by the short exact sequence $0\to E[2]\to E\to E\to 0$.  Moreover, define the corestriction Selmer group to be
    \[\sel{\sC(K)}(\QQ,E[2])=\set{x\in H^1(\QQ, E[2])~:~\res_v(x)\in \sC_v(K/\QQ;E)\ \forall v\in\places_\QQ},\]
    where $\res_v:H^1(\QQ,E[2])\to H^1(\QQ_v,E[2])$ is the restriction map.
\end{definition}
\begin{proposition}\label{prop:genus theory bounds seldim}
    Let $K/\QQ$ be a multiquadratic extension.  Then there exists a function $D_K:\set{\textnormal{elliptic curves }E/\QQ}\to \ZZ$ such that:
    \begin{itemize}
        \item if $E(\QQ)[2]=0$ then $D_K(E)=0$;
        \item when $K$ is fixed, $\abs{D_K(E)}$ is uniformly bounded for all $E/\QQ$;
        \item $\dim \sel{2}(E/K)\geq g_2(K/\QQ;E)+D_K(E)$; and
        \item $\dim \sel{2}(E/K)\leq \braces{[K:\QQ]-1}\braces{g_2(K/\QQ;E)+D_K(E)}+[K:\QQ]\dim \sel{\sC(K)}(\QQ, E[2])$.
    \end{itemize}
\end{proposition}
\begin{proof}
    Let $G=\gal(K/\QQ)$.  Since $E$ is defined over $\QQ$, the Selmer group $\sel{2}(E/K)$ is an $\FF_2[G]$-module, and we use this structure to prove the claim.  We begin by considering the fixed space $\sel{2}(E/K)^G$.  

    Via the inflation-restriction exact sequence, writing $\res:H^1(\QQ,E[2])\to H^1(K,E[2])^G$ we have an exact sequence
    \[0\to H^1(K/\QQ,E(K)[2])\to \res^{-1}\sel{2}(E/K)\to \sel{2}(E/K)^G\to H^2(K/\QQ,E(K)[2]).\]
    Hence there exists $D_K(E)\in\ZZ$ which satisfies the first two points, such that
    \[\dim \sel{2}(E/K)^G=\dim \res^{-1}\sel{2}(E/K)+D_K(E).\]
    In particular, by \cite{Paterson2021}*{Lemmas 4.9(ii) and 4.11}
    \begin{equation}\label{eq:dimfs}
    \dim \sel{2}(E/K)^G=\dim \sel{\sC(K)}(\QQ,E[2])+g_2(K/\QQ;E)+D_K(E).
    \end{equation}
    This immediately provides the claimed lower bound.  Moreover, denoting the norm element by $N_{K/\QQ}:=\sum_{g\in G}g\in\FF_2[G]$, we recall that by \cite{Paterson2021}*{Lemma 4.11(i)}
    \begin{equation}\label{eq:NIneq}
    \dim N_{K/\QQ}\sel{2}(E/K)\leq \dim\sel{\sC(K)}(\QQ,E[2]).
    \end{equation}
    Now, by a standard representation theoretic argument (see \Cref{lem:dim of F2Gmod}), 
    \begin{align*}\label{eq:ineqs}
    \dim \sel{2}(E/K)
    &\leq (\#G-1)\dim \sel{2}(E/K)^{G}+N_{G}\cdot \sel{2}(E/K).
    \end{align*}
    Combining this with \Cref{eq:NIneq,eq:dimfs} we obtain the claimed upper bound.
\end{proof}

\subsection{Proof of Explicit \Cref{thm:INTRO MQ Selmer Avg} (Assuming \Cref{thm:INTRO MQGenusTheoryAverage})}
In order to state the explicit version of \Cref{thm:INTRO MQ Selmer Avg}, we will need to acquire some more notation.
\begin{definition}\label{def:L_v}
    For every multiquadratic extension $K/\QQ$ and each prime number $p\geq 5$ define local factors
    \[L_p(K):=\begin{cases}
        \frac{(p-1)(p^4-p^3+p^2-p+1)(46p^5+62p^4+79p^3+84p^2+84p+48)}{48(p^{10}-1)} & \substack{\textnormal{if }K/\QQ\textnormal{ is ramified and}\\\textnormal{quadratic at }p,}\\
        \frac{16p^{11}+16p^{10}-8p^9+8p^8-8p^7-10p^6-4p^5+7p^4-p^3-8p^2-24p-1}{16(p^{10}-1)(p+1)} & \substack{\textnormal{if }K/\QQ\textnormal{ is unramified and}\\\textnormal{quadratic at }p,}\\
        \frac{(p-1)(p^4-p^3+p^2-p+1)(5p^5+15p^4+13p^3+9p^2+13p+8)}{8(p^{10}-1)} &\substack{\textnormal{if }K/\QQ\textnormal{ is biquadratic at }p,}\\
        1&\substack{\textnormal{if }K/\QQ\textnormal{ is totally split at }p.}
    \end{cases}\]
    Moreover, define an archimedean factor
    \[L_{\infty}(K):=\begin{cases}
        \frac{1}{2}&\textnormal{if }K\textnormal{ is real,}\\
        \frac{9}{20}&\textnormal{if }K\textnormal{ is imaginary.}
    \end{cases}\]
\end{definition}

\begin{theorem}[Explicit version of \Cref{thm:INTRO MQ Selmer Avg}]\label{thm:MQ Selmer Avg}
    Let $K/\QQ$ be a multiquadratic extension.  Then 
    \[\sum_{\substack{v\in\places_\QQ\\v\nmid 6}}\cG_v(K)\leq A^-(K)\leq A^+(K)\leq \braces{[K:\QQ]-1}\sum_{v\in\places_\QQ}\cG_v(K) + 4[K:\QQ]\prod_{\substack{v\in\places_\QQ\\v\nmid 6}}L_v(K).\]
\end{theorem}
\begin{proof}
    Via the Hilbert irreducibility theorem, so that $E(\QQ)[2]=0$ for $100\%$ of $E/\QQ$, and \Cref{prop:genus theory bounds seldim}, we need only lower bound the average of the genus theory and upper bound that and the average dimension of $\sel{\sC(K)}(\QQ,E[2])$.  The lower bound is simply \Cref{thm:INTRO MQGenusTheoryAverage}.  Similarly, the upper bound holds by \cite{PatQTRV}*{Corollary 6.10}.
\end{proof}

\begin{corollary}[\Cref{thm:INTRO MQ Selmer Avg}]\label{cor:MQ Selmer Avg}
    Let $K/\QQ$ be a multiquadratic extension.  Then 
    \[\sum_{\substack{v\in\places_\QQ\\v\nmid 6}}\cG_v(K)\leq A^-(K)\leq A^+(K)\leq \braces{[K:\QQ]-1}\sum_{v\in\places_\QQ}\cG_v(K) + \bigo{\braces{\frac{23}{24}}^{\omega\braces{\Delta_K}}}.\]
\end{corollary}
\begin{proof}
    This is immediate from \Cref{thm:MQ Selmer Avg}, as $\prod_{v}L_v(K)=\bigo{(23/24)^{\omega(\Delta_K)}}$.
\end{proof}

\section{Averages of Local Sums on Ekedahl-type Sets}\label{sec:EkedahlAverages}
In this section we prove a general averaging result for sums of $p$-adic functions on lattice points.  The material here does not rely on the rest of the article, though it will be applied later to our problem of interest.  For the duration of this section, we fix an integer $n\geq 1$.  The boxes in which we count lattice points will be bounded using a fairly general height.
\begin{definition}
    For $\bk\in\ZZ_{>0}^n$, $\bc\in\RR_{>0}^n$, and $X>0$ we write
    \[V_{\bk,\bc}(X):=\prod_{i=1}^n\left[-(c_iX)^{1/k_i},\,(c_iX)^{1/k_i} \right]\subset \RR^n.\]
\end{definition}
We will also make use of a well known result of Davenport.
\begin{lemma}[Davenport's Lemma \cite{MR0043821}]\label{lem:Davenport}
Let $\mathcal{R}\subseteq \RR^n$ be a compact semialgebraic set.  Then
\[\#\ZZ^n\cap\mathcal{R}=\textup{Vol}(\mathcal{R})+O\braces{\max\set{\Vol(\bar{\mathcal{R}}),1}},\]
where $\textup{Vol}(\bar{\mathcal{R}})$ denotes the greatest volume of any projection of $\mathcal{R}$ onto a $d$-dimensional coordinate hyperplane over all such hyperplanes and all $d\in\set{1,\dots,n-1}$.
\end{lemma}

\subsection{The Ekedahl Sieve}
Our interest will be in averages of local functions over sieved lattice points.  Our sieved sets will be as follows.
\begin{definition}\label{def:EkedahlAdmissible}
    Let $\cU=(\cU_p)_{p\textnormal{ prime}}$ be a sequence where $\cU_p\subseteq \ZZ_p^n$ is a measurable subset with boundary of measure $0$.  Then for $\bk\in\ZZ_{>0}^n$, $\bc\in\RR_{>0}^n$, and $X>0$,
    \[V_{\bk,\bc}^\cU(X):=\set{\bx\in V_{\bk,\bc}(X)\cap\ZZ^n~:~\bx\not\in \cU_p\ \forall p}.\]
    Further, we say that the triple $(\cU,\bk,\bc)$ is Ekedahl-admissible if
    \[\lim_{Y\to\infty}\limsup_{X\to\infty}\frac{\#\set{\bx\in V_{\bk,\bc}(X)\cap\ZZ^n~:~\bx\in \cU_p\ \exists p>Y}}{\Vol(V_{\bk,\bc}(X))}=0.\]
\end{definition}

\begin{proposition}[\cite{cremona2021local}*{Proposition 3.4}, see also \cite{MR1740984}]\label{prop:EkedahlSieve}  Let $(\cU,\bk,\bc)$ be an Ekedahl-admissible triple.  Then
    \[\lim_{X\to\infty}\frac{\#V_{\bk,\bc}^\cU(X)}{\Vol(V_{\bk,\bc}(X))}=\prod_{p\textnormal{ prime}}\braces{1-\mu_p(\cU_p)}.\]
\end{proposition}
\begin{rem}
    In the cited work, this is proved with $\bc=(1,\dots,1)$, however the proof of the general case follows mutatis mutandis.
\end{rem}
\subsection{Average at a Fixed Prime}

We first find the average of a single function.
\begin{lemma}\label{lem:fixed p average}
    Let $p$ be a prime, and assume that $\phi_p:\ZZ_p^n\to\RR$ is a bounded function which is locally constant outwith some closed set of measure zero.  Then for every Ekedahl-admissible tuple $(\cU,\bk,\bc)$,
    \[\lim_{X\to\infty}\frac{\sum\limits_{\bx\in V_{\bk,\bc}^{\cU}(X)}\phi(\bx)}{\#V_{\bk,\bc}^{\cU}(X)}=\frac{\int_{\ZZ_p^n\backslash \cU_p}\phi(\bz)d\bz}{1-\mu_p(\cU_p)}.\]
\end{lemma}
\begin{proof}
    Since $\phi$ is bounded and locally constant, there is an increasing (resp. decreasing) sequence of functions $\psi_0^-\leq \psi_1^-\leq\dots$ (resp. $\psi_0^+\geq \psi_1^+\geq\dots$) which is bounded above (resp. below) and converges to $\phi$ on a set of measure $1$, and such that each $\psi_{m}^\pm$ is defined by congruence conditions modulo $p^m$.  For ease, we write 
    \[\cU_{p,m}:=\im(\cU_p\to (\ZZ/p^m\ZZ)^n)\]
    Now, note for each choice of $\pm$ and $m\geq 0$, applying Davenport's lemma (\Cref{lem:Davenport}) and the Ekedahl sieve as in \Cref{prop:EkedahlSieve} we have
    \begin{align*}
        \lim_{X\to\infty}\frac{\sum\limits_{\bx\in V_{\bk,\bc}^\cU(X)}\psi_{m}^\pm(\bx)}{\#V^{\cU}_{\bk,\bc}(X)}
        &=\sum_{\substack{\ba\in(\ZZ/p^m\ZZ)^n\\\ba\not\in \cU_{p,m}}}\psi^\pm_m(\ba)\lim_{X\to\infty}\frac{\#\set{\bx\in V_{\bk,\bc}^\cU(X)~:~\bx\equiv \ba\mod p^m}}{\#V^{\cU}_{\bk,\bc}(X)}
        \\&=\frac{\int_{\ZZ_p^n\backslash \cU_{p,m}}\psi_m^\pm(\bz)d\bz}{1-\mu_p(\cU_p)}.
    \end{align*}
    Hence, for all $m\geq 0$
    \begin{align*}
        \limsup_{X\to\infty}\frac{\sum\limits_{\bx\in V_{\bk,\bc}(X)}\phi(\bx)}{\#V^{\cU}_{\bk,\bc}(X)}
        &\leq \lim_{X\to\infty}\frac{\sum\limits_{\bx\in V_{\bk,\bc}(X)}\psi_{m}^+(\bx)}{\#V_{\bk,\bc}^{\cU}(X)}
        =\frac{\int_{\ZZ_p^n\backslash\cU_{p,m}}\psi^+_m(\ba)d\ba}{1-\mu_p(\cU_p)}.
    \end{align*}
    In particular, taking a limit as $m\to\infty$ and using dominated convergence together with our assumption that $\cU_p$ has boundary of measure $0$, we obtain
    \[\limsup_{X\to\infty}\frac{\sum\limits_{\bx\in V_{\bk,\bc}(X)}\phi(\bx)}{\#V_{\bk,\bc}^{\cU}(X)}\leq \frac{\int_{\ZZ_p^n\backslash\cU_p}\phi(\ba)d\ba}{1-\mu_p(\cU_p)}.\]
    Similarly,
    \begin{align*}
        \liminf_{X\to\infty}\frac{\sum\limits_{\bx\in V_{\bk,\bc}(X)}\phi(\bx)}{\#V_{\bk,\bc}^{\cU}(X)}
        &\geq \lim_{X\to\infty}\frac{\sum\limits_{\bx\in V_{\bk,\bc}(X)}\psi_{m}^-(\bx)}{\#V_{\bk,\bc}^{\cU}(X)}
        = \frac{\int_{\ZZ_p^n\backslash \cU_{p,m}}\psi^-_m(\ba)d\ba}{1-\mu_p(\cU_p)}.
    \end{align*}
    Once again taking a limit as $m\to\infty$ we obtain the supporting $\liminf$ and so the claim holds.
\end{proof}

\subsection{Tail Estimate}
We will require a tail estimate in order to sum the previous result over all $p$.

\begin{lemma}\label{lem:unif bound congruence count}
    Let $(\cU,\bk,\bc)$ be an Ekedahl-admissible triple, $f_1,\dots,f_n\in\ZZ_{\neq 0}$, and consider a diagonal form $f(\bx)=\sum_{i=1}^nf_ix_i^{a_i}$.  Then, uniformly for all $p\nmid f_1$ and $X>0$, we have
    \[\frac{\#\set{\bx\in V_{\bk,\bc}^\cU(X)\cap\ZZ^n~:~f(\bx)\equiv 0\mod p^{k_1}}}{\#V^{\cU}_{\bk,\bc}(X)}\ll \frac{1}{p^{k_1}}+\frac{1}{X^{1/k_1}}.\]
\end{lemma}
\begin{proof}
    Clearly it is sufficient to prove the claim with $\cU_p=\emptyset$ for all $p$ and the constants $\bc=(1,\dots,1)$.  We freeze the final $n-1$ variables and apply the congruence to $x_1$:
    \begin{align*}
    &\#\set{\bx\in V_{\bk,\bc}^{\cU}(X)\cap\ZZ^n~:~f(\bx)\equiv 0\mod p^{k_1}}
    \\&\leq\sum_{\substack{(x_i)_{i=2}^n\in\ZZ^{n-1}\\\abs{x_i}^{k_i}\leq c_iX}}\#\set{\abs{x_1}\leq (c_1X)^{1/k_1}~:~f_1x_1^{a_1}\equiv -\sum_{i=2}^nf_ix_i^{a_i}\mod p^{k_1}}.
    \end{align*}
    Note that for every $a\in\ZZ/p^{k_1}\ZZ$ we can estimate
    \[\#\set{\abs{x_1}\leq X^{1/k_1}~:~f_1x_1^{a_1}\equiv a\mod p^{k_1}}\ll\frac{X^{1/k_1}}{p^{k_1}}+1,\]
    where the constant here depends on $n$ but, not on $p$.  Backfeeding to the sum, and using the Ekedahl sieve from \Cref{prop:EkedahlSieve},
    \begin{align*}
    \frac{\#\set{\bx\in V_{\bk,\bc}^{\cU}(X)\cap\ZZ^n~:~f(\bx)\equiv 0\mod p^{k_1}}}{\#V^{\cU}_{\bk,\bc}(X)}
    &\ll\frac{1}{{\#V^{\cU}_{\bk,\bc}(X)}}\sum_{\substack{(x_i)_{i=2}^n\in\ZZ^{n-1}\\\abs{x_i}^{k_i}\leq c_iX}}\braces{\frac{(c_1X)^{1/k_1}}{p^{k_1}}+1}
    \\&\ll\frac{1}{p^{k_1}}+\frac{1}{X^{1/k_1}},
    \end{align*}
    as required.
\end{proof}

\begin{lemma}\label{lem:tail est}
    Let $(\cU,\bk,\bc)$ be an Ekedahl-admissible triple, $f_1,\dots,f_n\in\ZZ_{\neq 0}^{n}$, and consider a diagonal form $f(\bx)=\sum_{i=1}^nf_ix_i^{a_i}$.  Assume that $k_1\geq 2$, and let $Y>0$ be larger than every prime dividing $f_1$.  Then

    \[\sum_{\substack{p>Y\\\textnormal{prime}}}\frac{\#\set{\bx\in V_{\bk,\bc}^\cU(X)\cap\ZZ^n~:~f(\bx)\equiv 0\mod p^{k_1}}}{\#V^{\cU}_{\bk,\bc}(X)}\ll Y^{1-k_1}+ \frac{1}{\log(X)}.
    \]
    uniformly in $Y$ and $X$.
\end{lemma}
\begin{proof}
    Our assumptions on $Y$ ensure by \Cref{lem:unif bound congruence count} that uniformly for all $p>Y$ and all $X$
    \[\frac{\#\set{\bx\in V_{\bk,\bc}^\cU(X)\cap\ZZ^n~:~f(\bx)\equiv 0\mod p^{k_1}}}{\#V^{\cU}_{\bk,\bc}(X)}\ll \frac{1}{p^{k_1}}+\frac{1}{X^{1/k_1}}.\]
    Thus the result follows, since $k_1\geq 2$ and so via an integral estimate and the prime number theorem
    \begin{align*}
    \sum_{\substack{p>Y\\\textnormal{prime}}}\frac{\#\set{\bx\in V_{\bk,\bc}^\cU(X)\cap\ZZ^n~:~f(\bx)\equiv 0\mod p^{k_1}}}{\#V^{\cU}_{\bk,\bc}(X)}
    &\ll \sum_{\substack{X^{1/k_1}\gg p>Y\\\textnormal{prime}}}\braces{\frac{1}{p^{k_1}}+\frac{1}{X^{1/k_1}}}\\
    &\ll Y^{1-k_1}+ \frac{1}{\log(X)},
    \end{align*}
    as required.
\end{proof}

\subsection{Local Sums}
We are now in a position to prove our main result.  Firstly, we need a definition to describe which local functions we can combine.
\begin{definition}\label{def:acceptable function}
    Let $(\cU,\bk,\bc)$ be an Ekedahl-admissible triple, and for every prime number $p$ let $\phi_p:\ZZ_p^n\to \RR_{>0}$.  We say that the collection $\phi=(\phi_p)_{p\textnormal{ prime}}$ is acceptable for $(\cU,\bk,\bc)$ if all of the following holds.
    \begin{enumerate}
        \item For each prime $p$, the function $\phi_p:\ZZ_p\to \RR_{>0}$ is locally constant outwith a closed set of measure zero.
        \item The supremum $\sup\set{\phi_p(\bx)~:~\forall\,p\textnormal{ prime, and }\bx\in\ZZ_p^n}<\infty$.
        \item The sum $\sum\limits_{p\textnormal{ prime}}\braces{\frac{\int_{\ZZ_p^n\backslash\cU_p}\phi_p(\bz)d\bz}{1-\mu_p(\cU_p)}}$ converges.
        \item There exist a diagonal form $f=\sum_{i=1}^nf_iX_i^{a_i}\in\ZZ[X_1,\dots,X_n]$, and $j\in\set{1,\dots,n}$ such that:
        \begin{enumerate}
        \item $f_j\neq 0$, and 
        \item $k_j\geq 2$, and
        \item for sufficiently large $p$, if $p^{k_j}\nmid f(\bx)$ then $\phi_p(\bx)=0$.
        \end{enumerate}
    \end{enumerate}
\end{definition}

\begin{theorem}[\Cref{thm:INTRO average of acceptable function over all lattice points}]\label{thm:average of acceptable function over all lattice points}
    Let $(\cU,\bk,\bc)$ be an Ekedahl-admissible triple, and $\phi=(\phi_p)_p$ be acceptable for this triple.  Then
    \[\lim_{X\to\infty}\frac{
    \sum\limits_{\bx\in V_{\bk,\bc}^\cU(X)}\braces{\sum\limits_p\phi_p(\bx)}
    }{
    \#V^{\cU}_{\bk,\bc}(X)
    }=\sum_{p\textnormal{ prime}}\frac{\int_{\ZZ_p^n\backslash \cU_p}\phi_p(\bz)d\bz}{1-\mu_p(\cU_p)}.\]
\end{theorem}
\begin{proof}
    For ease of notation, we will write $\phi(\bx)=\sum_p\phi_p(\bx)$.  For each $Y>0$ we split
    \begin{align*}
    \frac{\sum\limits_{\bx\in V_{\bk,\bc}(X)}\phi(\bx)}{\#V_{\bk,\bc}^{\cU}(X)} 
    &= \frac{\sum\limits_{p\leq Y}\sum\limits_{\bx\in V_{\bk,\bc}(X)}\phi_p(\bx)}{\#V_{\bk,\bc}^{\cU}(X)} + \frac{\sum\limits_{p>Y}\sum\limits_{\bx\in V_{\bk,\bc}(X)}\phi_p(\bx)}{\#V_{\bk,\bc}^{\cU}(X)}.
    \end{align*}
    Assume that $Y>0$ is sufficiently large (in terms of $f$ and $\phi$), we then apply \Cref{lem:fixed p average} to obtain
    \begin{align*}
        \liminf_{X\to\infty}\frac{\sum\limits_{\bx\in V_{\bk,\bc}(X)}\phi(\bx)}{\#V_{\bk,\bc}^{\cU}(X)} \geq \sum_{p\leq Y}\frac{\int_{\ZZ_p^n\backslash \cU_p}\phi_p(\bz)d\bz}{1-\mu_p(\cU_p)}.
    \end{align*}
    Taking a limit in $Y$ we then immediately obtain the claimed limit as a lower bound.  For the limsup we then use our tail bound.  We apply \Cref{lem:tail est,lem:fixed p average} to obtain
    \begin{align*}
        \limsup_{X\to\infty}\frac{\sum\limits_{\bx\in V_{\bk,\bc}(X)}\phi(\bx)}{\#V_{\bk,\bc}^{\cU}(X)} \leq \sum_{p\leq Y}\frac{\int_{\ZZ_p^n\backslash \cU_p}\phi_p(\bz)d\bz}{1-\mu_p(\cU_p)} + O\braces{\frac{1}{Y}},
    \end{align*}
    where the implied constant comes only from $\sup\set{\phi_p(\bx)~:~p\textnormal{ prime and }\bx\in\ZZ_p}$, which is finite by assumption, and the uniform constant in the cited lemma.  Taking a limit in $Y$ we obtain
    \begin{align*}
        \limsup_{X\to\infty}\frac{\sum\limits_{\bx\in V_{\bk,\bc}(X)}\phi(\bx)}{\#V_{\bk,\bc}^{\cU}(X)} 
        &\leq \sum_{p\textnormal{ prime}}\frac{\int_{\ZZ_p^n\backslash \cU_p}\phi_p(\bz)d\bz}{1-\mu_p(\cU_p)}
        \leq \liminf_{X\to\infty}\frac{\sum\limits_{\bx\in V_{\bk,\bc}(X)}\phi(\bx)}{\#V_{\bk,\bc}^{\cU}(X)} ,
    \end{align*}
    and so the limit exists and is equal to the claimed value.
\end{proof}

From this count, and the Ekedahl sieve, we then have the following corollary.

\begin{corollary}\label{cor:avg in Ekedahl sets}
    Let $(\cU,\bk,\bc)$ be an Ekedahl-admissible triple, and $\phi=(\phi_p)_p$ be acceptable for this triple.  Then
    \[\lim_{X\to\infty}\frac{
    \sum\limits_{\bx\in V_{\bk,\bc}^\cU(X)}\braces{\sum\limits_p\phi_p(\bx)}
    }{
    \#V_{\bk,\bc}^{\cU}(X)
    }=\sum_{p\textnormal{ prime}}\frac{\int_{\ZZ_p^n\backslash \cU_p}\phi(\bz)d\bz}{1-\mu_p(\cU_p)}.\]
\end{corollary}
\begin{proof}
    This is immediate from \Cref{thm:average of acceptable function over all lattice points} and \Cref{prop:EkedahlSieve}.
\end{proof}

\section{Local Norm Indices}\label{sec:norm indices}
We will now compute local norm indices associated to elliptic curves over multiquadratic extensions.  Throughout this section, we assume that $F/\QQ_\ell$ is a finite extension, and the residue characteristic satisfies $\ell\geq 5$.

In order to ensure clarity, we will use $v_F,\pi_F,k_F$ for the normalised valuation on $F$, a choice of uniformiser (fixed now and for the rest of the section) for $F$, and the residue field of $F$.  Similarly, for any finite extension $K/F$ we will write $v_K,\pi_K,k_K$ for the same data associated to $K$.

\subsection{The Tamagawa Ratio}
We firstly define a Tamagawa ratio and describe its behaviour, which will have relevance to the local norm index later on.
\begin{definition}\label{def:tamagawaratio}
    For every elliptic curve $E/F$, and every multiquadratic extension $K/F$, we define the Tamagawa ratio
    \[\cT(K/F;E):=\frac{\prod_{d\in S}c(E_d/F)}{c(E/K)}\]
    where $S=\ker(F^\times/F^{\times2}\to K^\times/K^{\times2})$.
\end{definition}
We will now compute these ratios in all cases, postponing the explanation of their utility to later.
\begin{definition}\label{def:minimal integral model}
    For an elliptic curve $E/F$, a minimal integral model is a small Weierstrass model
    \[E:y^2=x^3+Ax+B\]
    such that $A,B\in \cO_F$ and that either $v_F(A)<4$ or $v_F(B)<6$.
\end{definition}
\begin{proposition}\label{prop:unram tamratios at least 5}
    Let $K/F$ be the unramified quadratic extension.  Let $E/F$ be an elliptic curve, and
    \[E:y^2=x^3+Ax+B\]
    be a minimal integral model.  Then the Tamagawa ratio $\cT(K/F;E)$ is given by \Cref{tab:norm idx at unram primes}.
    \begin{table}[ht]
        \centering
        \begin{tabular}{|c|c|c|}
            \hline
            \multicolumn{3}{|c|}{$K/F$ an unramified quadratic extension, $E/F$ an elliptic curve,}\\
            \multicolumn{3}{|c|}{$E:y^2=x^3+Ax+B$ a minimal integral model.}\\
            \hline
            \textnormal{Kodaira Type of} $E/F$&Extra Condition&$\cT(K/F;E)$\\
            \hline
            $I_0$
                &-&$1$\\
            \hline
            \multirow{2}{*}{$I_{n>0}$}
                &$n$ \textnormal{even}&$2$\\
                &$n$ \textnormal{odd} &$1$\\
            \hline
            $II$
                &-&$1$\\
            \hline
            $III$
                &-&$2$\\
            \hline
            $IV$
                &-&$1$\\
            \hline
            \multirow{2}{*}{$I_0^*$}
                &$T^3+A\pi_F^{-2}T+B\pi_F^{-3}$ \textnormal{has }$3$\textnormal{ roots in }$k_F^{\times2}$&$4$\\\cline{2-3}         
                &\textnormal{otherwise}&$1$\\
            \hline
            \multirow{3}{*}{$I_{n>0}^*$}
                & $n$\textnormal{ even and }$-(27B^2+4A^3)\pi_F^{-(6+n)}\in k_F^{\times2}$&$4$\\
                & $n$\textnormal{ even and }$-(27B^2+4A^3)\pi_F^{-(6+n)}\not\in k_F^{\times2}$&$1$\\
                & $n$ \textnormal{odd} & $2$\\
            \hline
            $IV^*$
                &-&$1$\\
            \hline
            $III^*$
                &-&$2$\\
            \hline
            $II^*$
                &-&$1$\\
            \hline
        \end{tabular}
        \caption{Tamagawa ratio for unramified quadratic extensions.}\label{tab:norm idx at unram primes}
    \end{table}
\end{proposition}
\begin{proof}
    For ease of notation, let $K=F(\sqrt{u})$.  Note that minimal integral models for the other curves in the definition of $\cT(K/F;E)$ are given by
    \begin{align*}
        E_u/F&:y^2=x^3+Au^2x+Bu^3,\\
        E/K&:y^2=x^3+Ax+B.
    \end{align*}
    This then follows by a case analysis in Tate's algorithm (see \Cref{app:TatesAlg}).  In particular, note that the Kodaira types of $E/F$, $E/K$ and $E_u/F$ are all the same, and the only change can be in the splitness conditions.
    We list the cases below.
    \begin{itemize}
        \item If $E/F$ has type $I_0$, $II$, or $II^*$, then the Tamagawa numbers in the ratio are all $1$ and so $\cT(K/F;E)=1$.
        \item If $E/F$ has type $IV$ or $IV^*$, then since $u$ is nonsquare in $k_F$, precisely one of $E$ or $E_u$ has split subtype (see \Cref{app:TatesAlg}) over $F$ and the other is nonsplit, whilst the type of $E/K$ is automatically split.  Thus $\cT(K/F;E)=1$.
        \item If $E/F$ has type $III$ or $III^*$ then $\cT(K/F;E)=2$.
        \item If $E/F$ has type $I_0^*$ then write $P_E(T):=T^3+A\pi_F^{-2}T+B\pi_F^{-3}\in k_F[T]$ and $P_{E_u}(T):=T^3+Au^2\pi_F^{-2}T+Bu^3\pi_F^{-3}\in k_F[T]$.  Note that there is a bijection between the roots of $P_E$ and $P_{E_u}$ given by $\alpha\mapsto u\alpha$, and so
        \[\cT(K/F;E)=\frac{(1+\#\set{\alpha\in k_F~:~P_E(\alpha)=0})^2}{(1+\#\set{\alpha\in k_K~:~P_E(\alpha)=0})}.\]
        If $P$ is a product of linear factors over $k_F$ then immediately $\cT(K/F;E)=4$.  If $P(T)$ is irreducible over $k_F$ then it is also over $k_K$ (which is a quadratic extension of $k_F$), so $\cT(K/F;E)=1$.  Finally if $P(T)$ is a product of a linear and quadratic irreducible factor, then the quadratic factor splits over $k_K$ and so $\cT(K/F;E)=1$.
        \item If $E/F$ has multiplicative reduction of type $I_n$, then precisely one of $E/F$ or $E_u/F$ has split multiplicative reduction, with the other being nonsplit.  Moreover, $E/K$ must have split reduction of type $I_n$.  Thus $\cT(K/F;E)=2$ if $n$ is even, and $\cT(K/F;E)=1$ otherwise.
        \item If $E/F$ has potentially multiplicative reduction of type $I_n^*$, we break into cases depending on the parity of $n$.  If $n$ is even, then either both $E/F$ and $E_u/F$ have split $I_n^*$ reduction or both have nonsplit $I_n^*$ reduction.  Moreover, $E/K$ necessarily has split $I_n^*$ reduction.  Thus $\cT(K/F;E)=4$ if $E/F$ is split (i.e. $-(27B^2+4A^3)/\pi_F^{n+6}\in k_F^{\times2}$) and $\cT(K/F;E)=1$ otherwise.  If, on the other hand, $n$ is odd, then it is clear that precisely one of $E/F$ or $E_u/F$ has split $I_n^*$ reduction and the other must have nonsplit $I_n^*$ reduction.  As in the even case, the reduction type is necessarily split over $K$.  Thus we have that $\cT(K/F;E)=2$.
    \end{itemize}
\end{proof}

\begin{lemma}\label{lem:Kodaira types of ramified twists}
    Let $K=F(\sqrt{\theta})$ be a ramified quadratic extension.  Let $E/F$ be an elliptic curve.  The Kodaira types of $E_\theta/F$ and $E/K$ are determined by that of $E/F$ and $\theta$, and are listed in \Cref{tab:Kodaira types for ram twists and extns}.
\end{lemma}
\begin{table}[ht]
    \centering
    \begin{tabular}{|c|c|c|}
        \hline
        \multicolumn{3}{|c|}{Kodaira Types}\\
        \hline
        $E/F$&$E_\theta/F$&$E/K$\\
        \hline
        $I_{n\geq 0}$&$I_n^*$&$I_{2n}$\\
        $II$&$IV^*$&$IV$\\
        $III$&$III^*$&$I_0^*$\\
        $IV$&$II^*$&$IV^*$\\
        \hline
    \end{tabular}\caption{Kodaira types of ramified twists of elliptic curves}\label{tab:Kodaira types for ram twists and extns}
\end{table}
\begin{proof}
    Note that, since this is a ramified quadratic extension where the residue characteristic is odd, $v_F(\theta)$ is odd.  This is then a simple check using Tate's algorithm (see, e.g., \Cref{app:TatesAlg}).
\end{proof}

\begin{rem}
    Since the Kodaira types of $E/K$ and $E_\theta/K$ are the same, and quadratic twisting is an involution (on the level of isomorphism classes of curves), we need only list each Kodaira type as either that of $E/F$ or $E_\theta/F$.
\end{rem}

\begin{proposition}\label{prop:ram tamratios at least 5}
    Let $K/F$ be a ramified quadratic extension.  Let $\theta\in F$ be such that $K=F(\sqrt{\theta})$ and $v_F(\theta)=1$.  Let $E/F$ be an elliptic curve, and
    \[E:y^2=x^3+Ax+B\]
    be a minimal integral model.  Then the Tamagawa ratio $\cT(K/F;E)$ is given by \Cref{tab:norm idx at ram primes}.
    \begin{table}[ht]
        \begin{tabular}{|c|c|c|c|}
            \hline
            \multicolumn{4}{|c|}{$K=F(\sqrt{\theta})$ a ramified quadratic extension, $E/F$ an elliptic curve,}\\
            \multicolumn{4}{|c|}{$E:y^2=x^3+Ax+B$ a minimal integral model.}\\
            \hline
            \textnormal{Kodaira Type of} $E/F$&\multicolumn{2}{c|}{Extra Condition(s)}&$\cT(K/F;E)$\\
            \hline
            \multirow{3}{*}{$I_0$}
                &\multicolumn{2}{c|}{$T^3+AT+B$\textnormal{ has no roots in }$k_F$}&$1$\\
                &\multicolumn{2}{c|}{$T^3+AT+B$\textnormal{ has }$1$\textnormal{ root in }$k_F$}&$2$\\
                &\multicolumn{2}{c|}{$T^3+AT+B$\textnormal{ has }$3$\textnormal{ roots in }$k_F$}&$4$\\
            \hline
            \multirow{3}{*}{$I_0^*$}
                &\multicolumn{2}{c|}{$T^3+A\theta^{-2} T+B\theta^{-3}$\textnormal{ has no roots in }$k_F$}&$1$\\
                &\multicolumn{2}{c|}{$T^3+A\theta^{-2} T+B\theta^{-3}$\textnormal{ has }$1$\textnormal{ root in }$k_F$}&$2$\\
                &\multicolumn{2}{c|}{$T^3+A\theta^{-2} T+B\theta^{-3}$\textnormal{ has }$3$\textnormal{ roots in }$k_F$}&$4$\\
            \hline 
            \multirow{4}{*}{$I_{2n>0}$}
                &\multirow{2}{*}{$6B\in k_F^{\times2}$}
                    &$-(27B^2+4A^3)\theta^{-2n}\in k_F^{\times2}$&$2$\\
                    &&$-(27B^2+4A^3)\theta^{-2n}\not\in k_F^{\times2}$&$1$\\\cline{2-4}
                &\multirow{2}{*}{$6B\not\in k_F^{\times2}$}
                    &$-(27B^2+4A^3)\theta^{-2n}\in k_F^{\times2}$&$4$\\
                    &&$-(27B^2+4A^3)\theta^{-2n}\not\in k_F^{\times2}$&$2$\\
            \hline
            \multirow{4}{*}{$I_{2n>0}^*$}
                &\multirow{2}{*}{$6B\theta^{-3}\in k_F^{\times2}$}
                    &$-(27B^2+4A^3)\theta^{-2n-6}\in k_F^{\times2}$&$2$\\
                    &&$-(27B^2+4A^3)\theta^{-2n-6}\not\in k_F^{\times2}$ &$1$\\\cline{2-4}
                &\multirow{2}{*}{$6B\theta^{-3}\not\in k_F^{\times2}$}
                    &$-(27B^2+4A^3)\theta^{-2n-6}\in k_F^{\times2}$ &$4$\\
                    &&$-(27B^2+4A^3)\theta^{-2n-6}\not\in k_F^{\times2}$ &$2$\\
            \hline
            \multirow{2}{*}{$I_{2n+1}$}
                &\multicolumn{2}{c|}{$6B(27B^2+4A^3)\theta^{-2n-1}\in k_F^{\times2}$}&$2$\\
                &\multicolumn{2}{c|}{$6B(27B^2+4A^3)\theta^{-2n-1}\not\in k_F^{\times2}$}&$1$\\
            \hline
            \multirow{2}{*}{$I_{2n+1}^*$}
                &\multicolumn{2}{c|}{$6B(27B^2+4A^3)\theta^{-2n-10}\in k_F^{\times2}$}&$2$\\
                &\multicolumn{2}{c|}{$6B(27B^2+4A^3)\theta^{-2n-10}\not\in k_F^{\times2}$}&$1$\\
            \hline
            $II$, $II^*$, $IV$, $IV^*$&\multicolumn{2}{c|}{-}&$1$\\
            \hline
            \multirow{2}{*}{$III$}
                &\multicolumn{2}{c|}{$-A\theta^{-1}\not\in k_F^{\times2}$}&$2$\\
                &\multicolumn{2}{c|}{$-A\theta^{-1}\in k_F^{\times2}$}&$1$\\
            \hline
            \multirow{2}{*}{$III^*$}
                &\multicolumn{2}{c|}{$-A\theta^{-3}\not\in k_F^{\times2}$}&$2$\\
                &\multicolumn{2}{c|}{$-A\theta^{-3}\in k_F^{\times2}$}&$1$\\
            \hline
            
        \end{tabular}
        \caption{Tamagawa ratio for ramified quadratic extensions.}\label{tab:norm idx at ram primes}
    \end{table}
\end{proposition}
\begin{proof}
    This will follow from a case analysis and Tate's algorithm, using \Cref{lem:Kodaira types of ramified twists}.  If $E/F$ has Kodaira type $I_{n\geq 0}$, $II$, $III$ or $IV$ then minimal integral models for the other curves in the definition of $\cT(K/F;E)$ are given by
    \begin{align*}
        E_\theta/F&:y^2=x^3+A\theta^2x+B\theta^3,\\
        E/K&:y^2=x^3+Ax+B.
    \end{align*}
    Otherwise $E/F$ has Kodaira type $I_{n\geq 0}^*$, $II^*$, $III^*$ or $IV^*$, and so minimal integral models for the other curves in the definition of $\cT(K/F;E)$ are given by
    \begin{align*}
        E_\theta/F&:y^2=x^3+A\theta^{-2}x+B\theta^{-3},\\
        E/K&:y^2=x^3+A\theta^{-2}x+B\theta^{-3}.
    \end{align*} 
    With these models in mind, we now perform the case analysis.  The uniformisers that we use for Tate's algorithm over $F$ and $K$ will be $\pi_F=\theta$ and $\pi_K=\sqrt{\theta}$ respectively.
    \begin{itemize}
        \item If $E/F$ has type $I_0$ reduction, then $E_\theta/F$  has type $I_0^*$ and $E/K$ has type $I_0$.  Thus, by \Cref{app:TatesAlg} we have
        \[\cT(K/F;E)=1+\#\set{\alpha\in k_F~:~\alpha^3+A\alpha+B=0}.\]
        The case that $E/F$ has reduction type $I_0^*$ is similar.
        \item If $E/F$ has reduction type $I_n$ for some $n>0$ then by \Cref{lem:Kodaira types of ramified twists} $E_\theta/F$ has type $I_n^*$ and $E/K$ has type $I_{2n}$.  
        Moreover, the reduction type of $E/F$ is split if and only if that of $E/K$ is split (the residue fields satisfy $k_F=k_K$).  Thus by \Cref{app:TatesAlg} we have
        \[\frac{c(E/F)}{c(E/K)}=\begin{cases}
            1/2&\textnormal{if }E/F\textnormal{ has split reduction,}\\
            1/2&\textnormal{if }n\textnormal{ is odd and }E/F\textnormal{ has nonsplit reduction,}\\
            1&\textnormal{else.}
        \end{cases}\]
        The result in this case then follows by computing the Tamagawa number $c(E_\theta/F)$, as is shown in \Cref{app:TatesAlg}.

        The argument when $E/F$ has type $I_n^*$ reduction is similar, swapping the roles of $E$ and $E_\theta$.

        \item If $E/F$ has reduction type $II$, then $E_\theta/F$ has reduction type $IV^*$ and $E/K$ has type $IV$. The splitting conditions for $E/K$ and $E_\theta/F$ are equivalent (each is split if and only if $B\theta^{-1}\in k_F^{\times2}$), and so in particular one notes that $\cT(K/F;E)=1$.  Similarly, the cases where $E/F$ has reduction type $IV^*$, $IV$ or $II^*$ give $\cT(K/F;E)=1$.

        \item If $E/F$ has reduction type $III$ then $E_\theta/F$ has reduction type $III^*$ and $E/K$ has type $I_0^*$.  Moreover $v_K(A)=2$, $v_K(B)\geq 4$, so via Tate's algorithm \Cref{app:TatesAlg} we have
        \[\cT(K/F;E)=\frac{4}{1+\#\set{\alpha\in k_K~:~\alpha^3+(A/\theta)\alpha=0}}.\]
        so the result is as required.  Again, the proof for $E/F$ of type $III^*$ is similar by interchanging the roles of $E$ and $E_\theta$ above.
    \end{itemize}
    Having treated the case of each possible reduction type of $E/F$, the proof is complete.
\end{proof}

Since $F$ has odd residue characteristic, there is precisely one multiquadratic extension which is not accounted for by \Cref{tab:norm idx at unram primes,tab:norm idx at ram primes}:  the unique biquadratic extension.  We now provide the result there.

\begin{proposition}\label{prop:tamratio in biquad}
    Let $K/F$ be the biquadratic extension.  Write $K=F(\sqrt{u},\sqrt{\theta})$, where $u$ is a nonsquare unit in the integers of $F$ and $v_F(\theta)$ is odd.  Let $E/F$ be an elliptic curve, and
    \[E:y^2=x^3+Ax+B\]
    be a minimal integral model.  Then the Tamagawa ratio $\cT(K/F;E)$ is given by \Cref{tab:tamratio in biquad}.
\end{proposition}
\begin{proof}
    It is easy to see from the definition of the Tamagawa ratio that there is an equality
    \[\cT(K/F;E)=\cT(F(\sqrt{u})/F;E)\cdot\cT(F(\sqrt{u})/F;E_{\theta})\cdot\cT(K/F(\sqrt{u});E).\]
    Note that the reduction type of $E_\theta/F$ can be obtained from that of $E/F$ by applying \Cref{lem:Kodaira types of ramified twists}, and that the Kodaira type of $E/F(\sqrt{u})$ is the same as that of $E/F$ (with potential changes to splitting conditions).  Thus we can compute all of the terms on the right hand side of this equality by \Cref{prop:unram tamratios at least 5,prop:ram tamratios at least 5}, which provides the entries seen in \Cref{tab:tamratio in biquad}.
\end{proof}

\begin{table}[ht]
    \centering
        \begin{tabular}{|c|c|c|}
            \hline
            \multicolumn{3}{|c|}{$K/F$ the biquadratic extension, $E/F$ an elliptic curve,}\\
            \multicolumn{3}{|c|}{$E:y^2=x^3+Ax+B$ a minimal integral model.}\\
            \hline
            \textnormal{Kodaira Type of} $E/F$&Extra Condition(s)&$\cT(K/F;E)$\\
            \hline
            \multirow{3}{*}{$I_0$}
                &$T^3+AT+B$\textnormal{ has no roots in }$k_F$&$1$\\
                &$T^3+AT+B$\textnormal{ has }$1$\textnormal{ root in }$k_F$&$4$\\
                &$T^3+AT+B$\textnormal{ has }$3$\textnormal{ roots in }$k_F$&$16$\\
            \hline
            \multirow{3}{*}{$I_0^*$}
                &$T^3+A\theta^{-2} T+B\theta^{-3}$\textnormal{ has no roots in }$k_F$&$1$\\
                &$T^3+A\theta^{-2} T+B\theta^{-3}$\textnormal{ has }$1$\textnormal{ root in }$k_F$&$4$\\
                &$T^3+A\theta^{-2} T+B\theta^{-3}$\textnormal{ has }$3$\textnormal{ roots in }$k_F$&$16$\\
            \hline
            \multirow{2}{*}{$I_n$}
                &$n$\textnormal{ even and }$-(27B^2+4A^3)\theta^{-n}\in k_F^{\times2}$&$16$\\
                &otherwise&$4$\\
            \hline
            \multirow{2}{*}{$I_n^*$}
                &$n$\textnormal{ even and }$-(27B^2+4A^3)\theta^{-n-6}\in k_F^{\times2}$&$16$\\
                &otherwise&$4$\\
            \hline
            $II$, $II^*$, $IV$, $IV^*$ & & 1\\
            \hline
            $III$, $III^*$ & & 8\\
            \hline
        \end{tabular}
        \caption{Tamagawa ratio for the biquadratic extension of $F$.}\label{tab:tamratio in biquad}
    \end{table}

\subsection{Local Norm Index}
We now justify our interest in the Tamagawa ratio above.  It turns out to in fact be the local norm index.
\begin{proposition}\label{prop:norm index is tamratio}
    Let $K/F$ be a multiquadratic extension. For every elliptic curve $E/F$ we have
    \[\#E(F)/N_{K/F}E(K)=\mathcal{T}(K/F;E).\]
\end{proposition}
\begin{proof}
    To ease notation we write $G:=\gal(K/F)$, and $X:=\hom(G, \FF_2)$.  We write $\chi_0\in X$ for the trivial homomorphism, and for each $\chi\in X$ we write $\ZZ^{(\chi)}$ for the $\ZZ[G]$-module which is isomorphic to $\ZZ$ as an abelian group and on which $\sigma\in G$ acts by multiplication by $\chi(\sigma)$.  We will simply write $\ZZ$ for $\ZZ^{\chi_0}$.

    Consider the maps of $\ZZ[G]$-modules given by
    \begin{align*}
    \phi:\ZZ[G]&\to \bigoplus_{\chi\in X}\ZZ^{(\chi)}\\
    \sum_{\sigma}a_\sigma \sigma&\mapsto \left(\sum_{\sigma\in G}a_\sigma\chi(\sigma)\right)_{\chi\in X},
    \end{align*}
    and 
    \begin{align*}
    \hat{\phi}:\bigoplus_{\chi\in X}\ZZ^{(\chi)}&\to \ZZ[G]\\
    (b_\chi)_{\chi\in X}&\mapsto \sum_{\chi\in X}b_\chi\sum_{\sigma\in G}\chi(\sigma)\sigma.
    \end{align*}

Both $\phi\circ\hat{\phi}$ and $\hat{\phi}\circ\phi$ are multiplication by $\#G$ on the respective modules.  Thus we have a commutative diagram of $\ZZ[G]$-modules with exact rows given by:
\[\begin{tikzcd}
    0\ar{r}&
        \bigoplus\limits_{\chi\in X\backslash\set{\chi_0}}\ZZ^{(\chi)}
        \ar{r}{\hat{\phi}}\ar{d}{[\#G]}&
        \ZZ[G]
        \ar{r}{N}\ar{d}{\phi}&
        \ZZ
        \ar{r}{}\ar[equal]{d}&
        0\\
    0\ar{r}&
        \bigoplus\limits_{\chi\in X\backslash\set{\chi_0}}\ZZ^{(\chi)}
        \ar{r}{}&
        \bigoplus\limits_{\chi\in X}\ZZ^{(\chi)}
        \ar{r}{}&
        \ZZ
        \ar{r}{}&
        0,
\end{tikzcd}\]
where the map $N:\sum a_\sigma\sigma\mapsto \sum a_\sigma$ is given by action of the norm element of $\ZZ[G]$, and the maps on the bottom row are the natural inclusion and projection.  Via the twisting formalism of \cite{mazur2007twisting}*{Lemma 1.3, Lemma 2.3, Prop 4.1, Example 1.5(ii)}, this gives rise to a commutative diagram of abelian varieties with exact rows 
\begin{equation}\label{eq:commdiagramfornormindex}
\begin{tikzcd}
    0\ar{r}&
        \bigoplus\limits_{d\in S\backslash\set{1}}E_d
        \ar{r}{\hat{\phi}}\ar{d}{[\#G]}&
        \Res_{K/F}E
        \ar{r}{N}\ar{d}{\phi}&
        E
        \ar{r}{}\ar[equal]{d}&
        0\\
    0\ar{r}&
        \bigoplus\limits_{d\in S\backslash\set{1}}E_d
        \ar{r}{}&
        \bigoplus\limits_{d\in S}E_d
        \ar{r}{}&
        E
        \ar{r}{}&
        0,
\end{tikzcd}\end{equation}
where $\Res_{K/F}E$ is the Weil restriction, and we abuse notation by reusing $\phi,\hat{\phi}$ for now the corresponding isogenies of abelian varieties induced by the module maps above.  Explicitly, on $F$-points the the map $N$ acts on $\Res_{K/F}E(F)=E(K)$ as the norm map $N_{K/F}$.  Taking $F$-points above, noting that since the bottom right map is projection it remains surjective on $F$-points, we obtain a short exact sequence
\begin{equation}\label{eq:SESfornormindex}\begin{tikzcd}
0
\ar{r}&
\bigoplus\limits_{d\in S\backslash\set{1}}\frac{E_d(F)}{(\#G)E_d(F)}
\ar{r}&
\frac{\bigoplus\limits_{d\in S}E_d(F)}{\phi(\Res_{K/F}E(F))}
\ar{r}&
\frac{E(F)}{N_{K/F}E(K)}
\ar{r}&
0.
\end{tikzcd}\end{equation}
Using a result of Schaefer \cite{MR1370197}*{Lemma 3.8}, we can describe the order of the central term:
\begin{equation}\label{eq:SchaeferInitial}
\#\frac{\bigoplus\limits_{d\in S}E_d(F)}{\phi(\Res_{K/F}E(F))}=\frac{\abs{\phi'(0)}_F\prod_{d\in S}c(E_d/F)}{\#\Res_{K/F}E(F)[\phi]\cdot c(\Res_{K/F}E/F)},
\end{equation}
where $\abs{\phi'(0)}_F$ is the normalised absolute value of the determinant of the Jacobian matrix of partials of $\phi$ evaluated near $0$.  An elementary diagram chase in \Cref{eq:commdiagramfornormindex}, using that the rightmost vertical map is equality, we obtain that $\Res_{K/F}E[\phi]\cong \oplus_{d\in S\backslash\set{1}}E_d[\#G]$.  Moreover, by \cite{MR2961846}*{3.19} we have that $c(\Res_{K/F}E/F)=c(E/K)$, and so from \Cref{eq:SchaeferInitial} and \Cref{eq:SESfornormindex} we obtain
\begin{align*}
\#\frac{E(F)}{N_{K/F}E(K)}
&=\left(\prod\limits_{d\in S\backslash\set{1}}\frac{\#\tfrac{E_d(F)}{(\#G)E_d(F)}}{\#E_d(F)[\#G]}\right)\frac{\abs{\phi'(0)}_F\prod_{d\in S}c(E_d/F)}{c(E/K)}
\\&=\abs{\phi'(0)}_F\mathcal{T}(K/F;E),
\end{align*}
where the second equality uses that the residue characteristic is odd and each $E_d(F)$ contains a finite index subgroup isomorphic to the integers of $F$ (see, e.g., \cite{silverman2009arithmetic}*{VII Proposition 6.3}).  

It remains to show that $\abs{\phi'(0)}_F=1$, which we now do.  Note that $N_{K/F}E(K)\supseteq (\#G)E(F)$, and so the order of the norm index is a power of two, and the computations of \Cref{prop:unram tamratios at least 5,prop:ram tamratios at least 5,prop:tamratio in biquad} show that $\cT(K/F;E)$ is also a power of $2$.   On the other hand, $\abs{\phi'(0)}_F$ is an integer power of the residue characteristic, which is odd, and so must be $1$ in order for the displayed equation above to hold, concluding the proof.
\end{proof}

We will not actually be making use of the norm index all of the time, but in fact the norm index modulo $2$ which is the object appearing in the genus theory formula $g_2(K/F;E)$.  For quadratic extensions there is nothing to distinguish, but for biquadratic we have to be more careful.

\begin{proposition}\label{prop:quad norm indices at least 5}
    Let $K/F$ be a quadratic extension.  Let $E/F$ be an elliptic curve, and
    \[E:y^2=x^3+Ax+B\]
    be a minimal integral model.  Then we have an equality
    \[\#E(F)/\left(N_{K/F}E(K)+2E(F)\right) = \cT(K/F;E),\]
     and so the norm index modulo $2$ is given by: \Cref{tab:norm idx at unram primes} if $K/F$ is unramified; or \Cref{tab:norm idx at ram primes} if $K/F$ is ramified. 
\end{proposition}
\begin{proof}
    Clearly $E(F)/N_{K/F}E(K)$ is $[K:F]=2$-torsion and so this follows from \Cref{prop:norm index is tamratio} and: if $K/F$ is unramified \Cref{prop:unram tamratios at least 5} or if $K/F$ is ramified then \Cref{prop:ram tamratios at least 5}.
\end{proof}

For the biquadratic case we must be more careful.  First we will need a helpful lemma which is true in far more generality than it is presented but we will only require it in our present setting.
\begin{lemma}\label{lem:4torsion is 4cokernel}
Let $E/F$ be an elliptic curve.  Then there is an isomorphism of groups
\[E(F)[4]\cong E(F)/4E(F).\]   
\end{lemma}
\begin{proof}
    There is a finite index subgroup, arising from the filtration by formal groups, of $E(F)$ which is isomorphic to the additive group of integers $\cO_F$ of $F$ (see e.g. \cite{silverman2009arithmetic}*{VII Prop. 6.3}).  We will name this subgroup $E_1(F)$, and note that (since the residue characteristic of $F$ is coprime to $4$) we have $E_1(F)=4E_1(F)\subseteq 4E(F)$.  Since $E_1(F)$ has finite index in $E(F)$, we certainly have an isomorphism
    \begin{equation}\label{eq:isom1thing}
    \frac{E(F)}{E_1(F)}[4]\cong \frac{\frac{E(F)}{E_1(F)}}{4\frac{E(F)}{E_1(F)}}.
    \end{equation}
    Now consider the commutative diagram
    \[\begin{tikzcd}
        0\ar{r}&E_1(F)\ar{r}\ar{d}{\times 4}&E(F)\ar{r}\ar{d}{\times 4}&E(F)/E_1(F)\ar{r}\ar{d}{\times 4}&0\\
        0\ar{r}&E_1(F)\ar{r}&E(F)\ar{r}&E(F)/E_1(F)\ar{r}&0.
    \end{tikzcd}\]
    An application of the snake lemma, using the fact that multiplication by $4$ is bijective on $E_1(F)\cong \cO_F$, provides isomorphisms
    \begin{align*}
        E(F)[4]&\cong \frac{E(F)}{E_1(F)}[4]& E(F)/4E(F)\cong \frac{\frac{E(F)}{E_1(F)}}{4\frac{E(F)}{E_1(F)}}.
    \end{align*}
    Combining these with \Cref{eq:isom1thing} we obtain the result. 
\end{proof}
We can now deduce the required norm index modulo $2$ from the Tamagawa ratio.
\begin{proposition}\label{prop:nmidx in biquad}
    Let $K/F$ be the biquadratic extension.  Write $K=F(\sqrt{u},\sqrt{\theta})$, where $u$ is a nonsquare unit in the integers of $F$ and $v_F(\theta)$ is odd.  Let $E/F$ be an elliptic curve, and
    \[E:y^2=x^3+Ax+B\]
    be a minimal integral model.  Then the norm index modulo $2$, $\#\tfrac{E(F)}{\left(N_{K/F}E(K)+2E(F)\right)}$, is given by \Cref{tab:norm idx in biquad}.
\end{proposition}
\begin{table}[ht]
    \centering
        \begin{tabular}{|c|c|c|}
            \hline
            \multicolumn{3}{|c|}{$K/F$ the biquadratic extension, $E/F$ an elliptic curve,}\\
            \multicolumn{3}{|c|}{$E:y^2=x^3+Ax+B$ a minimal integral model.}\\
            \hline
            \textnormal{Kodaira Type of $E/F$}&Extra Condition(s)&$\#\tfrac{E(F)}{N_{K/F}E(K)+2E(F)}$\\
            \hline
            \multirow{3}{*}{$I_0$}
                &$T^3+AT+B$\textnormal{ has no roots in }$k_F$&$1$\\
                &$T^3+AT+B$\textnormal{ has }$1$\textnormal{ root in }$k_F$&$2$\\
                &$T^3+AT+B$\textnormal{ has }$3$\textnormal{ roots in }$k_F$&$4$\\
            \hline
            \multirow{3}{*}{$I_0^*$}
                &$T^3+A\theta^{-2} T+B\theta^{-3}$\textnormal{ has no roots in }$k_F$&$1$\\
                &$T^3+A\theta^{-2} T+B\theta^{-3}$\textnormal{ has }$1$\textnormal{ root in }$k_F$&$2$\\
                &$T^3+A\theta^{-2} T+B\theta^{-3}$\textnormal{ has }$3$\textnormal{ roots in }$k_F$&$4$\\
            \hline
            \multirow{2}{*}{$I_n$}
                &$n$\textnormal{ even and }$-(27B^2+4A^3)\theta^{-n}\in k_F^{\times2}$&$4$\\
                &otherwise&$2$\\
            \hline
            \multirow{2}{*}{$I_n^*$}
                &$n$\textnormal{ even and }$-(27B^2+4A^3)\theta^{-n-6}\in k_F^{\times2}$&$4$\\
                &otherwise&$2$\\
            \hline
            $II$, $II^*$, $IV$, $IV^*$ & & 1\\
            \hline
            $III$, $III^*$ & & 4\\
            \hline
        \end{tabular}
        \caption{Norm index modulo $2$ from the biquadratic extension of $F$.}\label{tab:norm idx in biquad}
    \end{table}

\begin{proof}
    Note firstly that, by \Cref{lem:4torsion is 4cokernel} and the fact that $4E(F)\subseteq N_{K/F}E(K)$, we can identify $E(F)/N_{K/F}E(K)$ as a quotient of a subgroup of the abelian group $(\ZZ/4\ZZ)^2$.  Considering \Cref{prop:norm index is tamratio} it is then clear that whenever $\cT(K/F;E)=1,2,8,16$ then $\#\left(E(F)/N_{K/F}E(K)+2E(F)\right)=1,2,4,4$ respectively.  Using \Cref{prop:tamratio in biquad} we can then fill in all of the cases aside from those for which $\cT(K/F;E)=4$, in which case we have two possibilities:
    \[E(F)/N_{K/F}E(K)\cong
    \begin{cases}
        \ZZ/4\ZZ&\textnormal{ or}\\
        \ZZ/2\ZZ\times \ZZ/2\ZZ,
    \end{cases}\]
    and these have different sizes modulo $2$.  The cases when $\cT(K/F;E)=4$ are when
    \begin{enumerate}
        \item\label{enum1:biquad nmidx} $E$ has Kodaira type $I_{n>0}$ or $I_{n>0}^*$, and the discriminant of the minimal integral model ($\Delta_E=-(27B^2+4A^3)$) satisfies $\Delta_E\not\in F^{\times2}$; or
        \item\label{enum2:biquad nmidx} $E$ has Kodaira type $I_0$ and $T^3+AT+B$ has $1$ root in $k_F$; or
        \item \label{enum3:biquad nmidx} $E$ has Kodaira type $I_0^*$ and $T^3+A\theta^{-2}T+B\theta^{-3}$ has $1$ root in $k_F$.
    \end{enumerate}
    Before we deal with each of these cases, note that it is enough to show that $E(F)$ does not have full $2$-torsion (i.e. $E(F)[2]\not\cong (\ZZ/2\ZZ)^2$)  Indeed by \Cref{lem:4torsion is 4cokernel} we would have $E(F)/4E(F)\cong E(F)[4]\subseteq \ZZ/4\ZZ$ and since $E(F)/N_{K/F}E(K)$ is a quotient of this group we obtain
    \[E(F)/N_{K/F}E(K)\cong \ZZ/4\ZZ,\]
    which would imply the remaining results in the table.

    For case \Cref{enum1:biquad nmidx}: the discriminant is nonsquare and as this is also the discriminant of the cubic polynomial $f(T)=T^3+AT+B$ (whose roots give the $2$-torsion points on $E$), we must have that the Galois group of $f$ is not a subgroup of $A_3$ so in particular contains an order $2$ element.  Thus $E(F)[2]$ cannot be full.

    For case \Cref{enum2:biquad nmidx} note that if $E(F)$ has full $2$-torsion then since we have good reduction so would the reduced curve (and so we would have $3$ roots, not $1$ over $k_F$).  Similarly, for case \Cref{enum3:biquad nmidx}, note that there is a bijection between the roots of $T^3+AT+B$ over $F$ and those of $T^3+A\theta^{-2}T+B\theta^{-3}$ (namely send $\alpha\mapsto \theta^{-1}\alpha$) and so again we cannot have full $2$-torsion as then we'd have $4$ roots in $k_F$ and not $1$.
\end{proof}

\section{Average of Genus Theory Invariant}\label{sec:avg of genus theory}

We will use our result on local sums in the Ekedahl sieve to compute the average of the genus theory invariant $g_2(K/\QQ;E_{A,B})$ for a multiquadratic field $K/\QQ$ as $(A,B)\in\Epsilon(X)$ varies.

\subsection{Local Densities}
We now compute some densities which will be of use in what follows.  We will frequently make use of Tate's algorithm, for which we have provided a reference table in \Cref{app:TatesAlg}.  To ease our space use somewhat, we introduce some notation for this section.

\begin{definition}
For each prime number $p$, let
\[\Epsilon_p:=\set{(A,B)\in\ZZ_p^2~:~\substack{\bullet 4A^3+27B^2\neq0\\\bullet (A,B)\not\in p^4\ZZ_p\times p^6\ZZ_p}}.\]
\end{definition}

\subsubsection{Primes of Type $III$}
\begin{lemma}\label{lem:dens for III and star}
    Let $p\geq 5$ be a prime number and $a\in\FF_p^\times$, then
    \[\mu_p\braces{\set{(A,B)\in\Epsilon_p~:~\substack{
    E_{A,B}\textnormal{ has reduction type }III\\\textnormal{and } Ap^{-1}\equiv a\mod p
    }}}=p^{-4},\]
    and
    \[\mu_p\braces{\set{(A,B)\in\Epsilon_p~:~\substack{
    E_{A,B}\textnormal{ has reduction type }III^*\\\textnormal{and } Ap^{-3}\equiv a\mod p
    }}}=p^{-9}.\]
\end{lemma}
\begin{proof}
    It follows from Tate's algorithm that for $(A,B)\in\Epsilon_p$, the curve $E_{A,B}$ is an elliptic curve with reduction type $III$ satisfying the required congruence condition if and only if 
    \[(A,B)\equiv (ap,0)\mod p^2.\]  
    The second equality is seen by noting that the set we are taking density of is the image of the one in the first equality under the map $(A,B)\mapsto (p^2A,p^3B)$ by Tate's algorithm.
\end{proof}
\subsubsection{Primes of Type $I_0$}
\begin{lemma}\label{lem:I_0 mod l counting with roots}
    For each prime number $p\geq 5$ and $n\in \set{0,1,3}$,
    \[\#\set{(a, b)\in\FF_p^2~:~\substack{T^3+aT+b\textnormal{ is separable}\\\textnormal{and has }n\textnormal{ roots}}}=\begin{cases}
        \frac{(p^2-1)}{3}&\textnormal{if }n=0;\\
        \frac{p(p-1)}{2}&\textnormal{if }n=1;\\
        \frac{(p-1)(p-2)}{6}&\textnormal{if }n=3.
\end{cases}\]
\end{lemma}
\begin{proof}
    For ease, we will write $P_{a,b}(T):=T^3+aT+b\in\FF_p[T]$.  Note that the $3$ roots $\set{\alpha_1,\alpha_2,\alpha_3}\subseteq \bar{\FF}_p$ of $P_{a,b}$ satisfy $\alpha_1+\alpha_2+\alpha_3=0$, because the $T^2$ coefficient in $P_{a,b}$ is $0$.
    
    Consider, first, the case $n=0$.  Here $P_{a,b}(T)$ is irreducible, and the set of irreducible monic cubic polynomials is in $1:3$ correspondence with elements $\alpha\in \FF_{p^3}\backslash \FF_p$.  Under this correspondence the polynomials with $T^2$ coefficient being $0$ (our set of $P_{a,b}(T)$) correspond to $\alpha$ with trace $0$.  Thus
    \begin{align*}
    &\#\set{(a, b)\in\FF_p^2~:~P_{(a,b)}(T)\textnormal{ is irreducible over }\FF_p}
    \\&=\frac{1}{3}\left(\#\ker(\operatorname{Tr}_{\FF_{p^3}/\FF_p})-1\right)
    \\&=\frac{p^2-1}{3},
    \end{align*}
    where we use that the trace is surjective (since $p\neq 3$, the only element of $\FF_p$ with trace $0$ is $0$).

    Now consider $n=1$.  In this case, $P_{a,b}(T)$ must factor as a product of one monic linear polynomial and one monic irreducible quadratic polynomial.  Moreover, since the $T^2$ coefficient is $0$, the root of the linear polynomial must be equal to $-\operatorname{Tr}_{\FF_{p^2}/\FF_p}(\alpha)$ where $\alpha$ is a root of the quadratic factor.  Thus
    \begin{align*}
    &\#\set{(a, b)\in\FF_p^2~:~P_{(a,b)}(T)\textnormal{ has }1\textnormal{ root in }\FF_p \textnormal{ and no repeated roots}}
    \\&=\#\set{(a',b')\in\FF_p^2~:~ T^2+a'T+b'\textnormal{ is irreducible}}
    \\&=\frac{1}{2}\#\FF_{p^2}\backslash\FF_p=\frac{p(p-1)}{2}.
    \end{align*}

    Finally, for the case $n=3$, it is elementary to see that $\#\Epsilon^{p}_{(1,1)}=p^2-p$, and so the result follows by subtracting the counts of the previous cases.
\end{proof}

\begin{lemma}\label{lem:dens for I0 and star with poly}
    Let $p\geq 5$ be a prime number and $n\in\set{0,1,3}$, then
    \[\mu_p\braces{\set{(A,B)\in\Epsilon_p~:~\substack{
        E_{A,B}\textnormal{ has reduction type }I_0\\\textnormal{and } T^3+AT+B\textnormal{ has }n\textnormal{ roots}
        }}}=\begin{cases}
        \frac{(p^2-1)}{3p^2}&\textnormal{if }n=0,\\
        \frac{(p-1)}{2p}&\textnormal{if }n=1,\\
        \frac{(p-1)(p-2)}{6p^2}&\textnormal{if }n=3,
\end{cases}\]
and
    \[\mu_p\braces{\set{(A,B)\in\Epsilon_p~:~\substack{
        E_{A,B}\textnormal{ has reduction type }I_0^*\\\textnormal{and } T^3+Ap^{-2}T+Bp^{-3}\textnormal{ has }n\textnormal{ roots}
        }}}=
        \begin{cases}
        \frac{(p^2-1)}{3p^7}&\textnormal{if }n=0,\\
        \frac{(p-1)}{2p^6}&\textnormal{if }n=1,\\
        \frac{(p-1)(p-2)}{6p^7}&\textnormal{if }n=3,
\end{cases}
    \]
\end{lemma}
\begin{proof}
    The first equality follows from \Cref{lem:I_0 mod l counting with roots}.  The second equality is seen by noting that the set we are taking the density of is the image of the one in the first equality under the map $(A,B)\mapsto (p^2A,p^3B)$ by \Cref{app:TatesAlg}.
\end{proof}
\subsubsection{Primes of Type $I_{n>0}$}
\begin{lemma}\label{lem:modulo counting for In}
    Let $p\geq 5$ be a prime number, $n\geq 1$ be an integer. For every $B\in(\ZZ/p^{n+1}\ZZ)^\times$, and $u\in \FF_p^\times$ we have that
    \[
    \#\set{A\in(\ZZ/p^{n+1}\ZZ)^\times~:~(4A^3+27B^2)=up^{n}}=
    \begin{cases}
    \#\mu_3(\FF_p)&\textnormal{if }(B^2\mod p)\in 4\FF_p^{\times3},\\
    0&\textnormal{else.}
    \end{cases}
    \]
\end{lemma}
\begin{proof}
    Immediate from Hensel lifting: since $p\geq 5$ and $\tfrac{1}{4}(up^n-27B^2)\in (\ZZ/p^{n}\ZZ)^\times$, the roots of $T^3-\tfrac{1}{4}(up^n-27B^2)$ are in bijection with those of its reduction mod $p$.
\end{proof}

\begin{lemma}\label{lem:dens for In and star double cdn}
    Let $p\geq 5$ be a prime number and $n>0$ an integer, and let $R_1,R_2\in\set{\FF_p^{\times2},\FF_p^\times\backslash\FF_p^{\times2}}$. Then
    \[\mu_p\braces{\set{(A,B)\in\Epsilon_p~:~\substack{
            E_{A,B}\textnormal{ is type }I_n\textnormal{ at }p\\
            (B\mod p)\in R_1\\
            ((4A^3+27B^2)/p^{n}\mod p)\in R_2
            }}}=\frac{(p-1)^2}{4p^{n+2}}\]
    and 
    \[\mu_p\braces{\set{(A,B)\in\Epsilon_p~:~\substack{
            E_{A,B}\textnormal{ is type }I_n^*\textnormal{ at }p\\
            (Bp^{-3}\mod p)\in R_1\\
            ((4A^3+27B^2)/p^{n+6}\mod p)\in R_2
            }}}=\frac{(p-1)^2}{4p^{n+7}}.\]
\end{lemma}
\begin{proof}
    By Tate's algorithm
    \begin{align*}
        \set{(A,B)\in\Epsilon_p~:~\substack{
            E_{A,B}\textnormal{ is type }I_n\textnormal{ at }p\\
            (B\mod p)\in R_1\\
            ((4A^3+27B^2)/p^{n}\mod p)\in R_2
            }}
        &=\set{(A,B)\in\ZZ_p^2~:~\substack{
            (B\mod p)\in R_1\\
            4A^3+27B^2\equiv 0\mod p^n\\
            ((4A^3+27B^2)/p^{n}\mod p)\in R_2
            }}.
    \end{align*}
    It follows from \Cref{lem:modulo counting for In} that for each $B\in\ZZ_p^\times$,
    \begin{align*}
    \int_{\substack{
            A\in\ZZ_p\\
            4A^3+27B^2\equiv 0\textup{ mod }p^n\\
            ((4A^3+27B^2)p^{-n}\textup{ mod }p)\in R_2
            }}dA
    =\begin{cases}
        \frac{(p-1)\#\mu_3(\FF_p)}{2p^{n+1}}&\textnormal{if }B^2\textup{ mod }p\in 4\FF_p^{\times3},\\
        0&\textnormal{else}.
    \end{cases}
    \end{align*}
    Thus
    \begin{align*}
        \int_{\substack{
            B\in\ZZ_p\\
            (B\textup{ mod }p)\in R_1
            }}
        \int_{\substack{
            A\in\ZZ_p\\
            v_p(4A^3+27B^2)=n\\
            ((4A^3+27B^2)p^{-n}\textup{ mod }p)\in R_2
            }}dAdB
        &=\int_{\substack{
            B\in\ZZ_p\\
            (B\textup{ mod }p)\in R_1\\
            (B^2\textup{ mod }p)\in 4\FF_p^{\times3}\\
            }}
            \frac{\#\mu_3(\FF_p)(p-1)}{2p^{n+1}}dB
        \\&=\frac{\#\FF_p^{\times6}}{\#\FF_p}
            \frac{\#\mu_3(\FF_p)(p-1)}{2p^{n+1}}
        \\&=\frac{\#\mu_3(\FF_p)(p-1)^2}{\#\mu_6(\FF_p)2p^{n+2}}
        \\&=\frac{(p-1)^2}{4p^{n+2}},
    \end{align*}
    as required.

    The second equality is seen by noting that the set we are taking the density of is the image of the one in the first equality under the map $(A,B)\mapsto (p^2A,p^3B)$ by Tate's algorithm.
\end{proof}

\subsection{Evaluating the Integral}
We now have calculate the densities of the sets where the norm index is constant, and use this to evaluate the integral for each multiquadratic extension $F/\QQ_p$,
\begin{equation}\label{eq:padic integral}
\int_{(A,B)\in\Epsilon_p}\dim_{\FF_2}\frac{E_{A,B}(\QQ_p)}{N_{F/\QQ_p}E_{A,B}(F)+2E_{A,B}(\QQ_p)}dAdB.
\end{equation}

\begin{proposition}\label{prop:local densities with const norm index}
    Let $p\geq 5$ be a prime number and $F/\QQ_p$ be a multiquadratic extension, and for brevity for each $i\geq 0$ write
    \[Z_F(i):=\set{(A,B)\in\Epsilon_p~:~\dim_{\FF_2}\frac{E_{A,B}(\QQ_p)}{N_{F/\QQ_p}E_{A,B}(F)+2E_{A,B}(\QQ_p)}=i}.\]
    Then the $p$-adic densities $\mu_p(Z_F(i))$ are given by \Cref{tab:MQDENS}.
\end{proposition}
\vspace{-10pt}
\begin{table}[h]
    \begin{tabular}{|c|c|c|c|}
    \hline
    &\multicolumn{3}{c|}{$F$}\\
    \cline{2-4}
    $\mu_p(Z_{F}(i))$&\multicolumn{2}{c|}{Quadratic}&Biquadratic\\
    \cline{2-3}
    &Unramified&Ramified&\\
    \hline
    $i=1$&$\frac{(p-1)(p^7+p^6+p^5+p^3+p+1)}{p^9(p+1)}$&$\frac{(p-1)(p^5+1)(p^3+p^2+1)}{2p^9}$&$\frac{(p^5+1)(p-1)(p^2+3p+1)}{2p^7(p+1)}$\\
    $i=2$&$\frac{(p-1)\braces{p^2-p+1}}{6p^7(p+1)}$&$\frac{(p-1)(p^5+1)(2p^2-2p-1)}{12p^7(p+1)}$&$\frac{(p^5+1)(p-1)\braces{p^4-p^3+p^2+6p+6}}{6p^9(p+1)}$\\
    \hline
    \end{tabular}
    \caption{For $p\geq 5$ and multiquadratic $F/\QQ_p$, the $p$-adic densities of the sets where $\phi_F=i$}\label{tab:MQDENS}
\end{table}
\vspace{-30pt}
\begin{proof}
    The proof splits into cases depending on $F$.

    \noindent\textbf{Case: $F$ is unramified.} By \Cref{prop:quad norm indices at least 5} and \Cref{tab:norm idx at unram primes}, we have
    \begin{align*}
        Z_F(1)&=\set{(A,B)\in\Epsilon_p~:~\substack{
        E_{A,B}\textnormal{ has reduction type given by one of the following:}\\
        \bullet I_n\textnormal{ for some }n\in 2\ZZ_{>0}\\
        \bullet III\\
        \bullet I_n^*\textnormal{ for some }n\in (2\ZZ_{\geq0}+1)\\
        \bullet III^*\\ 
        }},\\
        Z_F(2)&=\set{(A,B)\in\Epsilon_p~:~\substack{
        E_{A,B}\textnormal{ has reduction type given by one of the following:}\\
        \bullet I_0^*\textnormal{ and }T^3+Ap^{-2}T+Bp^{-3}\textnormal{ has }3\textnormal{ roots in }\FF_p\\
        \bullet I_n^*\textnormal{ for some }n\in 2\ZZ_{>0}\textnormal{ and }-(27B^2+4A^4)p^{-(6+n)}\in\FF_p^{\times2}\\ 
        }}.
    \end{align*}
    Therefore, using \Cref{lem:dens for In and star double cdn,lem:I_0 mod l counting with roots,lem:dens for III and star}, we have
    \begin{align*}
        \mu_p(Z_F(1))&=\sum_{n\geq 1}\braces{\frac{(p-1)^2}{p^{2n+2}}} + \frac{p-1}{p^4}+\sum_{n\geq 1}\braces{\frac{(p-1)^2}{p^{2n+6}}}+\frac{p-1}{p^9},\\
        &=\frac{(p^4+1)(p-1)^2}{p^6(p^2-1)}+\frac{(p^5+1)(p-1)}{p^9}\\
        &=\frac{(p-1)(p^7+p^6+p^5+p^3+p+1)}{p^9(p+1)},\\
        \mu_p(Z_F(2))&=\frac{(p-1)(p-2)}{6p^7}+\sum_{n\geq 1}\frac{(p-1)^2}{2p^{2n+7}}\\
        &=\frac{(p-1)(p-2)}{6p^7}+\frac{(p-1)}{2p^{7}(p+1)}\\
        &=\frac{(p-1)\braces{p^2-p+1}}{6p^7(p+1)}.
    \end{align*}

    \noindent\textbf{Case: $F$ is a ramified quadratic extension.}  By \Cref{prop:quad norm indices at least 5} and \Cref{tab:norm idx at ram primes}, writing $\theta\in\ZZ_p$ for an element such that $F=\QQ_p(\sqrt{\theta})$,
    \begin{align*}
        Z_F(1)&=\set{(A,B)\in\Epsilon_p~:~\substack{
        E_{A,B}\textnormal{ has reduction type given by one of the following:}\\
        \bullet I_0\textnormal{ and }T^3+AT+B\textnormal{ has }1\textnormal{ root in }\FF_p\\
        \bullet I_0^*\textnormal{ and }T^3+A\theta^{-2}T+B\theta^{-3}\textnormal{ has }1\textnormal{ root in }\FF_p\\
        \bullet I_n\textnormal{ for some }n\in \ZZ_{>0}\textnormal{ and }(-1)^{n+1}6B(4A^3+27B^2)\theta^{-n}\in \FF_p^{\times2}\\
        \bullet I_n^*\textnormal{ for some }n\in \ZZ_{>0}\textnormal{ and }(-1)^{n+1}6B(4A^3+27B^2)\theta^{-(n+6)}\in \FF_p^{\times2}\\
        \bullet III\textnormal{ and }-A\theta^{-1}\not\in\FF_p^{\times2}\\
        \bullet III^*\textnormal{ and }-A\theta^{-3}\not\in\FF_p^{\times2}\\
        }},\\
        Z_F(2)&=\set{(A,B)\in\Epsilon_p~:~\substack{
        E_{A,B}\textnormal{ has reduction type given by one of the following:}\\
        \bullet I_0\textnormal{ and }T^3+AT+B\textnormal{ has }3\textnormal{ roots in }\FF_p\\
        \bullet I_0^*\textnormal{ and }T^3+A\theta^{-2}T+B\theta^{-3}\textnormal{ has }3\textnormal{ roots in }\FF_p\\
        \bullet I_n\textnormal{ for some }n\in 2\ZZ_{>0}\textnormal{ and }6B\not\in\FF_p^{\times2}\textnormal{ and }(4A^3+27B^2)\theta^{-n}\in \FF_p^{\times2}\\
        \bullet I_n^*\textnormal{ for some }n\in 2\ZZ_{>0}\textnormal{ and }6B\theta^{-3}\not\in\FF_p^{\times2}\textnormal{ and }(4A^3+27B^2)\theta^{-(n+6)}\in \FF_p^{\times2}
        }},
    \end{align*}
    Therefore, using \Cref{lem:dens for In and star double cdn,lem:I_0 mod l counting with roots,lem:dens for III and star}, we have
    \begin{align*}
        \mu_p(Z_F(1))
        &=\frac{(p-1)(p^5+1)}{2p^6}+\sum_{n\geq1}\braces{\frac{(p-1)^2(p^5+1)}{2p^{n+7}}}+\frac{(p-1)(p^5+1)}{2p^9}\\
        &=\frac{(p-1)(p^5+1)}{2p^6}\braces{1+\frac{1}{p}+\frac{1}{p^3}}\\
        &=\frac{(p-1)(p^5+1)(p^3+p^2+1)}{2p^9},\\
        \mu_p(Z_F(2))
        &= \frac{(p-1)(p-2)(p^5+1)}{6p^7}+\sum_{n\geq 1}\frac{(p^5+1)(p-1)^2}{4p^{2n+7}}\\
        &= \frac{(p-1)(p^5+1)}{12p^7}\braces{2(p-2)+\frac{3(p-1)}{p^2-1}}\\
        &= \frac{(p-1)(p^5+1)(2p^2-2p-1)}{12p^7(p+1)}.
    \end{align*}

    \noindent\textbf{Case: $F$ is the biquadratic extension.}  By \Cref{prop:nmidx in biquad} and \Cref{tab:norm idx in biquad}
    \begin{align*}
        Z_{F}(1)&=\set{(A,B)\in\Epsilon_p~:~\substack{
        E_{A,B}\textnormal{ has reduction type given by one of the following:}\\
            \bullet I_0\textnormal{ and }T^3+AT+B\textnormal{ has }1\textnormal{ root in }\FF_p\\
            \bullet I_0^*\textnormal{ and }T^3+A\theta^{-2}T+B\theta^{-3}\textnormal{ has }1\textnormal{ root in }\FF_p\\
            \bullet I_n\textnormal{ for some }n\in 2\ZZ_{>0}\textnormal{ and }-(4A^3+27B^2)p^{-n}\not\in \FF_p^{\times2}\\
            \bullet I_n^*\textnormal{ for some }n\in 2\ZZ_{>0}\textnormal{ and }-(4A^3+27B^2)p^{-(n+6)}\not\in \FF_p^{\times2}\\
            \bullet I_n\textnormal{ or }I_n^*\textnormal{ for some }n\in (2\ZZ_{>0}-1)\\
        }},\\
        Z_F(2)&=\set{(A,B)\in\Epsilon_p~:~\substack{
        E_{A,B}\textnormal{ has reduction type given by one of the following:}\\
            \bullet I_0\textnormal{ and }T^3+AT+B\textnormal{ has }3\textnormal{ roots in }\FF_p\\
            \bullet I_0^*\textnormal{ and }T^3+A\theta^{-2}T+B\theta^{-3}\textnormal{ has }3\textnormal{ roots in }\FF_p\\
            \bullet I_n\textnormal{ for some }n\in 2\ZZ_{>0}\textnormal{ and }-(4A^3+27B^2)p^{-n}\in \FF_p^{\times2}\\
            \bullet I_n^*\textnormal{ for some }n\in 2\ZZ_{>0}\textnormal{ and }-(4A^3+27B^2)p^{-(n+6)}\in \FF_p^{\times2}\\
            \bullet III\textnormal{ or }III^* 
        }}.
    \end{align*}
    Therefore, using \Cref{lem:dens for In and star double cdn,lem:I_0 mod l counting with roots,lem:dens for III and star}, we have
    \begin{align*}
        \mu_p(Z_F(1))
        &=\frac{(p^5+1)(p-1)}{2p^6}+\sum_{n\geq 1}\braces{\frac{(p^5+1)(p-1)^2}{2p^{2n+7}}}+\sum_{n\geq 1}\frac{(p^5+1)(p-1)^2}{p^{2n+6}}\\
        &=\frac{(p^5+1)(p-1)}{2p^6}\braces{1+\frac{1}{p(p+1)}+\frac{2}{p+1}}\\
        &=\frac{(p^5+1)(p-1)(p^2+3p+1)}{2p^7(p+1)},\\
        \mu_p(Z_F(2))
        &=\frac{(p^5+1)(p-1)(p-2)}{6p^7}+\sum_{n\geq 1}\braces{\frac{(p^5+1)(p-1)^2}{2p^{2n+7}}}+\frac{(p^5+1)(p-1)}{p^9}\\
        &=\frac{(p^5+1)(p-1)}{6p^7}\braces{(p-2)+\frac{3}{p+1}+\frac{6}{p^2}}\\
        &=\frac{(p^5+1)(p-1)\braces{p^4-p^3+p^2+6p+6}}{6p^9(p+1)}\\
    \end{align*}
\end{proof}
\noindent We now evaluate the integral.
\begin{proposition}\label{prop:pAdicGenusIntegral}
    Let $p\geq 5$ be a prime number, $F/\QQ_p$ be a multiquadratic extension, then
    \begin{align*}
        &\int_{(A,B)\in\Epsilon_p}\dim_{\FF_2}\frac{E_{A,B}(\QQ_p)}{N_{F/\QQ_p}E_{A,B}(F)+2E_{A,B}(\QQ_p)}dAdB
        \\&\quad\quad=\begin{cases}
            \frac{(p-1)(3p^7+3p^6+3p^5+p^4+2p^3+p^2+3p+3)}{3p^9(p+1)}
            &\textnormal{if }F/\QQ_p\textnormal{ is unramified;}\\
            \frac{(p-1)(p^5+1)(5p^4+4p^3+2p^2+3p+3)}{6p^9(p+1)}
            &\textnormal{if }F/\QQ_p\textnormal{ is ramified and quadratic;}\\
            \frac{(p-1)(p^5+1)(5p^4+7p^3+5p^2+12p+12)}{6p^9(p+1)}&\textnormal{if }F/\QQ_p\textnormal{ is biquadratic.}
        \end{cases}
    \end{align*}
\end{proposition}
\begin{proof}
    By \Cref{prop:quad norm indices at least 5,prop:nmidx in biquad}, the integrand can only take the values $0,1,2$.  Thus, in the language of \Cref{prop:local densities with const norm index}, by definition
    \[\int_{(A,B)\in\Epsilon_p}\dim_{\FF_2}\frac{E_{A,B}(\QQ_p)}{N_{K/\QQ_p}E_{A,B}(K)+2E_{A,B}(\QQ_p)}dAdB=\mu_p(Z_F(1))+2\mu_p(Z_F(2)),\]
    and we compute the right hand side via loc. cit. (see \Cref{tab:MQDENS}).
\end{proof}

\subsection{The Archimedean Contribution}\label{subsec:Archimedean Contribution}

In order to determine the average behaviour of $g_2(K/\QQ;E)$ we have so so far computed the contribution from the finite places using the Ekedahl sieve.  It remains to determine the contribution coming from the archimedean place.  We firstly record a characterisation the norm index.
\begin{lemma}[\cite{kramer1981arithmetic}*{Proposition 6}]\label{lem:infinite places norm index}
    Let $E/\RR$ be an elliptic curve, write $\Delta_E$ for the discriminant of a choice of Weierstrass model for $E/\RR$.  Then
    \[\dim_{\FF_2}\braces{E(\RR)/N_{\CC/\RR}E(\CC)}=\begin{cases}
        1& \textnormal{if }\Delta_E>0;
        \\0&\textnormal{else.}
    \end{cases}\]
\end{lemma}

For each $(A,B)\in\Epsilon(X)$, $\Delta_{E_{A,B}}=-16(4A^3+27B^2)$, and so we need to count the number of elements $(A,B)\in\Epsilon(X)$ such that $4A^3+27B^2<0$.  We will do so using Dirichlet convolution.  We begin by proving some lemmata on related Dirichlet series before going on to apply these to obtain the average.

\begin{definition}\label{def:multfns}
    Let $f,g:\NN\to \RR$ be the multiplicative functions defined on prime powers by
    \begin{align*}
    f(p^r)&=\begin{cases}
        1-p^{-4}&\textnormal{if }r\geq 6,\\
        1&\textnormal{else;}
    \end{cases}&
    g(p^r)=\begin{cases}
        p^4&\textnormal{if }r\geq 6,\\
        1&\textnormal{else.}
    \end{cases}
    \end{align*}
\end{definition}

We will give asymptotics and estimates for these, as they will turn up in our average calculation for the Archimedean contribution to genus theory below.

\begin{lemma}\label{lem:DirSer f asymptotic}
    Let $\alpha\in\QQ_{\geq0}$.  Then for every real number $Y\geq 1$,
    \[\sum_{B=1}^Y B^\alpha f(B)=\braces{1-\frac{\alpha}{\alpha+1}}\zeta(10)^{-1}Y^{\alpha+1} +\bigo{Y^{\alpha}},\]
    where the implied constant depends on $\alpha$ but not on $Y$.
\end{lemma}
\begin{proof}
    We begin with the case $\alpha=0$.  Note that the Dirichlet convolution $\hat{f}:=f*\mu$ is the multiplicative function defined on prime powers by
    \[\hat{f}(p^r)=\begin{cases}
        -p^{-4}&\textnormal{if }r=6;\\
        1&\textnormal{if }r=0;\\
        0&\textnormal{else.}
    \end{cases}\]
    Thus
    \begin{equation}\label{eq:dirseries equation}
    \sum_{B=1}^Y f(B)=\sum_{B=1}^Y \sum_{b\mid B}\hat{f}(b)=Y\sum_{b=1}^Y \frac{\hat{f}(b)}{b}+\bigo{\sum_{b=1}^Y \abs{\hat{f}(b)}}.
    \end{equation}
    We compare the main term to the related Euler product via an integral estimate,
    \[\sum_{b=1}^Y \frac{\hat{f}(b)}{b}=\sum_{\substack{d=1\\\textnormal{sqfree}}}^{Y^{1/6}}\frac{(-1)^{\omega(d)}}{d^{10}}=\zeta(10)^{-1}+\bigo{Y^{-3/2}},\]
    and for the error in \Cref{eq:dirseries equation} we have
    \[\sum_{b=1}^Y \abs{\hat{f}(b)}\leq \sum_{d=1}^{Y^{1/6}}d^{-4}\ll 1.\]
    Thus we have $\sum_{B=1}^Yf(B)=\zeta(10)^{-1}Y+\bigo{1}$ as required.  For $\alpha\geq0$, we apply Abel's summation formula and the previous special case:
    \begin{align*}
    \sum_{B=1}^Y B^\alpha f(B)
    &=Y^{\alpha}\sum_{B=1}^Y f(B) - \int_{1}^Y\braces{\sum_{n=1}^uf(u)}\alpha u^{\alpha-1}du
    \\&= \zeta(10)^{-1}Y^{\alpha+1} - \zeta(10)^{-1}\frac{\alpha}{\alpha+1} Y^{\alpha+1} +\bigo{Y^{\alpha}},
    \end{align*}
    as required.
\end{proof}

\begin{lemma}\label{lem:DirSer g asymptotic}
    For every real number $Y\geq1$,
    \[\sum_{B=1}^Yg(B)=\braces{\prod_p\braces{1+\frac{p^4-1}{p^6}}}Y+\bigo{Y^{5/6}}.\] 
\end{lemma}
\begin{proof} Note that the Dirichlet convolution $\hat{g}:=g*\mu$ is the multiplicative function defined on prime powers by
    \[\hat{g}(p^r)=\begin{cases}
        p^4-1&\textnormal{if }r=6;\\
        1&\textnormal{if }r=0;\\
        0&\textnormal{else}.
    \end{cases}\]
    Thus
    \[\sum_{B=1}^Yg(B)=\sum_{B=1}^Y\sum_{b\mid B}\hat{g}(b)=Y\sum_{b=1}^Y\frac{\hat{g}(b)}{b}+\bigo{\sum_{b=1}^Y\abs{\hat{g}(b)}}.\]
    We compare the main term to the related Euler product via an integral estimate,
    \[\sum_{b=1}^Y \frac{\hat{g}(b)}{b}=\sum_{\substack{d=1\\\textnormal{sqfree}}}^{Y^{1/6}}\prod\limits_{\substack{p\mid d}}\frac{(p^4-1)}{p^6}=\prod_{p}\braces{1+\frac{p^4-1}{p^6}}+\bigo{Y^{-1/6}},\]
    Whereas the error satisfies
    \[\sum_{b=1}^Y\abs{\hat{g}(b)}\leq \sum_{\substack{d=1\\\textnormal{sqfree}}}^{Y^{1/6}}d^4\ll Y^{5/6},\]
    and so the result follows.
\end{proof}

Armed with the asymptotics above, we now compute the average contribution to genus theory from the Archimedean place.

\begin{proposition}\label{prop:archimedean count}
    For $X\geq 1$
    \[\frac{\#\set{(A,B)\in\Epsilon(X)~:~4A^3+27B^2<0}}{\#\Epsilon(X)}=\frac{1}{10}+\bigo{X^{-1/3}}.\]
\end{proposition}
\begin{proof}
    Ignoring pairs $(A,B)$ for which $4A^3+27B^2=0$ leads to the count 
    \begin{align*}
        &\#\set{(A,B)\in\Epsilon(X)~:~4A^3+27B^2<0}
        \\&=\sum_{\abs{B}\leq \sqrt{\frac{X}{27}}}\#\set{-\sqrt[3]{\frac{X}{4}}\leq A\leq -\sqrt[3]{\frac{27B^2}{4}}~:~(\forall p\textnormal{ prime})\ p^6\mid B\implies p^4\nmid A}+\bigo{X^{1/3}}
        \\&=2\sum_{1\leq B\leq \sqrt{\frac{X}{27}}}\braces{\braces{\sqrt[3]{\frac{1}{4}}X^{1/3}-\frac{3}{\sqrt[3]{4}}B^{2/3}}f(B)+\bigo{g(B)}} + \bigo{X^{1/3}},
    \end{align*}
    where $f,g$ are the multiplicative functions of \Cref{def:multfns}.  We then apply \Cref{lem:DirSer f asymptotic,lem:DirSer g asymptotic} to obtain estimates
    \begin{align*}
        \sqrt[3]{\frac{1}{4}}X^{1/3}\sum_{1\leq B\leq \sqrt{\frac{X}{27}}}f(B)
        &=\frac{X^{5/6}}{\sqrt[3]{4}\sqrt{27}}\zeta(10)^{-1}+\bigo{X^{1/3}}\\
        \frac{3}{\sqrt[3]{4}}\sum_{1\leq B\leq \sqrt{\frac{X}{27}}}B^{2/3}f(B)
        &=\frac{X^{5/6}}{\sqrt[3]{4}\sqrt{27}}\zeta(10)^{-1}\frac{3}{5}+\bigo{X^{1/3}}\\
        \sum_{1\leq B\leq \sqrt{\frac{X}{27}}}g(B)
        &\ll X^{1/2}.\\
    \end{align*}
    Putting these together, $\#\set{(A,B)\in\Epsilon(X)~:~4A^3+27B^2<0}=\frac{1}{10}\frac{4X^{5/6}}{\zeta(10)\sqrt[3]{4}\sqrt{27}}+\bigo{X^{1/2}}$, and so dividing by $\#\Epsilon(X)$ (determined via \Cref{prop:EkedahlSieve}) we obtain
    \[\frac{\#\set{(A,B)\in\Epsilon(X)~:~4A^3+27B^2<0}}{\#\Epsilon(X)}=\frac{1}{10}+\bigo{X^{-1/3}}.\]
\end{proof}

\subsection{Averaging the Genus Theory}
For completeness, we show the elementary claim that the family of elliptic curves is cut out by Ekedahl admissible conditions.
\begin{notation}\label{not:translation from EpsilonX to VkcU}
    For this section, $\bk=(3,2)$, $\bc=(1/4,1/27)$, and $\cU=(\cU_p)_p$ where for each prime $p$ we take $\cU_p=\ZZ_p^2\backslash\Epsilon_p$.  Note that for all $X>0$,
    \[\Epsilon(X)=V_{\bk,\bc}^{\cU}(X).\]
    Further, let $K/\QQ$ be a multiquadratic field.  For each prime number $p$, we choose a prime $w$ of $K$ lying over $p$.  If $p\geq 5$ and $(A,B)\in\ZZ_p^2$, we will write
    \[\varphi_p(A,B):=\begin{cases}
    \dim_{\FF_2}\frac{E_{A,B}(\QQ_p)}{N_{K_w/\QQ_p}E_{A,B}(K_w)+2E_{A,B}(\QQ_p)}&\textnormal{ if }(A,B)\in\Epsilon_p,\\
    0&\textnormal{else.}
    \end{cases}\]
    If $p\in\set{2,3}$, we set $\varphi_{p}(A,B)=0$.  We then have the collection $\varphi=(\varphi_p)_p$.
\end{notation}

\begin{lemma}\label{lem:EpsilonX is ekedahl}
    With notation as in \Cref{not:translation from EpsilonX to VkcU},  the triple given by $(\cU, \bk, \bc)$ is Ekedahl-admissible.
\end{lemma}
\begin{proof}
    Evidently $\cU_p$ is measurable, indeed its measure is $p^{-10}$, and its boundary is contained in the set of $(A,B)$ with $4A^3+27B^2=0$ which has measure $0$.

    Note that for $\bx\in V_{\bk,\bc}(X)$, if $\bx\in \cU_p$ then $p^4\leq (X/4)^{1/3}$ and $p^{6}\leq (X/27)^{1/2}$.  In particular, $p\leq (X/27)^{1/12}$.  Thus for $X>0$, applying Davenport's lemma (\Cref{lem:Davenport})
    \begin{align*}
    &\#\set{\bx\in V_{\bk,\bc}(X)\cap\ZZ^n~:~\bx\in \Epsilon_p,\ \exists p>Y}
    \\&\hspace{20pt}\leq \sum_{\substack{Y<p\leq(X/27)^{1/12}\\\textnormal{prime}}}\#\set{\bx\in V_{\bk,\bc}(X)\cap\ZZ^n~:~\bx\in \Epsilon_p}
    \\&\hspace{20pt}= \sum_{\substack{Y<p\leq(X/27)^{1/12}\\\textnormal{prime}}}\braces{p^{-10}\Vol(V_{\bk,\bc}(X))+\bigo{X^{1/2}}}.
    \end{align*}
    In particular, dividing by the volume and taking the limit
    \[\limsup_{X\to\infty}\frac{\#\set{\bx\in V_{\bk,\bc}(X)\cap\ZZ^n~:~\bx\in \Epsilon_p,\ \exists p>Y}}{\Vol(V_{\bk,\bc}(X))}
    \leq\sum_{Y<p}p^{-10}\ll Y^{-9},\]
    which has limit $0$ as $Y\to\infty$, so the triple $(\cU,\bk,\bc)$ is Ekedahl admissible.
\end{proof}

\begin{lemma}\label{lem:local norm is acceptable}
    With notation as in \Cref{not:translation from EpsilonX to VkcU}, the collection $\varphi$ is acceptable for $(\cU,\bk,\bc)$.
\end{lemma}
\begin{proof}
    We must check the conditions of \Cref{def:acceptable function}.  Firstly, note that the set $S_p=\set{(A,B)\in\ZZ_p^2~:~4A^3+27B^2=0}$ is a closed set of measure $0$.  Then our functions are locally constant by definition for $p\in\set{2,3}$, and by \Cref{prop:quad norm indices at least 5,prop:nmidx in biquad} for $p\geq 5$.  Moreover, since the local norm index is uniformly bounded (e.g. by \cite{Paterson2021}*{Lemma 5.3}), $\sup\set{\varphi_p(\bx)~:~\bx\in\ZZ}<\infty$.

    The sum 
    \[\sum_p\frac{\int_{(A,B)\in\Epsilon_p}\varphi_p(A,B)dAdB}{\mu_p(\Epsilon_p)}=\sum_p\frac{\int_{(A,B)\in\Epsilon_p}\varphi_p(A,B)dAdB}{1-p^{-10}}\]
    converges by \Cref{prop:pAdicGenusIntegral}, since for unramified $p$ the summand is $\bigo{p^{-2}}$.  It remains to produce the diagonal form $f$ from (4) of \Cref{def:acceptable function}.  In this case, let $f(A,B)=4A^3+27B^2$ be the (scaled) discriminant, so that for $j=2$ we have $f_j=27\neq 0$, $k_j=2\geq 2$.  Moreover, it follows from \Cref{prop:nmidx in biquad,prop:quad norm indices at least 5} that for unramified $p\geq 5$, if $E_{A,B}$ has reduction type $I_0$ or $I_1$ then $\varphi_p(A,B)=0$, and it is immediate then from Tate's algorithm (see \Cref{app:TatesAlg}) that these are the only reduction types for which $p^2\nmid f(A,B)$, and so the final axiom holds.
\end{proof}

Having now verified that we can apply our averaging result from \Cref{sec:EkedahlAverages}, we do so.

\begin{theorem}\label{thm:MQGenusTheoryAverage}
For every multiquadratic extension $K/\QQ$, let
\[\cG(K;X):=\frac{\sum\limits_{(A,B)\in\Epsilon(X)}g_2(K/\QQ;E_{A,B})}{\#\Epsilon(X)},\]
and write $\cG^+(K):=\limsup_{X\to\infty}\cG(K;X)$, and $\cG^-(K):=\liminf_{X\to\infty}\cG(K;X)$. Then, with notation as in \Cref{def:GvK Invariants},
\[\sum_{\substack{v\in\places_\QQ\\v\nmid 6}}\cG_v(K)\leq \cG^-(K)\leq \cG^+(K)\leq \sum_{v\in\places_\QQ}\cG_v(K).\]
\end{theorem}
\begin{proof}
Note that by definition, for every $(A,B)\in\Epsilon(X)$,
\begin{equation}\label{eq:genus expression for Ekedahl}
g_2(K/\QQ;E_{A,B}) = \sum_{\substack{p\geq 5}}\varphi_p(A,B) + \dim_{\FF_2}\frac{E_{A,B}(\RR)}{N_{K_\infty/\RR}E_{A,B}(K_\infty)}+\sum_{p\in\set{2,3}} \dim_{\FF_2}\frac{E_{A,B}(\QQ_p)}{N_{K_w/\QQ_p}E_{A,B}(K_w)}.
\end{equation}
With notation as in \Cref{not:translation from EpsilonX to VkcU}, we apply \Cref{cor:avg in Ekedahl sets} to the local functions in $\varphi$ over the Ekedahl-admissible triple $(\cU,\bk,\bc)$, together with \Cref{prop:pAdicGenusIntegral} this implies
\[\lim_{X\to\infty}\frac{\sum\limits_{(A,B)\in\Epsilon(X)}\sum\limits_{p\geq 5}\varphi_p(A,B)}{\#\Epsilon(X)}=\sum_{p\geq 5}\cG_v(K).\]
Moreover, it follows from \Cref{prop:archimedean count} and \Cref{lem:infinite places norm index} that the average of the archimedean term in \Cref{eq:genus expression for Ekedahl} is $\cG_\infty(K)$ as required.  Finally, that the terms at $2$ and $3$ are bounded as required is obtained by noting that dimensions are bounded by $3$ and $2$ at $2$ and $3$ respectively \cite{Paterson2021}*{Lemma 5.3}.
\end{proof}

\bibliography{refs}\addresseshere
\appendix
\section{A Useful Representation Theoretic Lemma}
Below we present a useful representation theoretic lemma, which, whilst necessary, does not fit well with the subject of the article.  Unfortunately we have been unable to find a reference for this, and so include it for completeness.
    \begin{lemma}\label{lem:dim of F2Gmod}
        Let $G$ be an abelian group of order $2^r$ for some $r>0$, and $M$ be a finite $\FF_2[G]$-module.  Then, writing $N_G:=\sum_{g\in G}g\in\FF_2[G]$
        \[\dim_{\FF_2} M\leq (2^r-1)\dim_{\FF_2}M^{G}+\dim_{\FF_2}\braces{N_G\cdot M}.\]
    \end{lemma}
    \begin{proof}
        We write $\#G=2^r$, and induct on $r$.  For $r=1$ the kernel of the norm map is the fixed space, and so the claimed bound is in fact an equality.  Now let $r>1$, and assume that the claimed inequality holds for abelian groups of order $2^{r-1}$.  Then combining the solutions for $G_1$ and $G/G_1$ we obtain
        \begin{align*}
            \dim M&=\dim M^{G_1}+\dim N_{G_1}\cdot M\\
            &\leq(2^{r-1}-1)\braces{\dim M^{G} +\dim \braces{N_{G_1}\cdot M}^{G/G_1}}+ \dim N_{G/G_1}M^{G_1}+\dim N_{G}\cdot M\\
            &\leq(2^{r}-1)\dim M^{G} +\dim N_{G}\cdot M,
        \end{align*}
        as required.
    \end{proof}

\begin{landscape}
\section{Tate's Algorithm}\label{app:TatesAlg}
Let $F$ be the completion of a number field at a non-archimedean place with residue characteristic $p\geq 5$.  Let $\cO_F$, $v_F$, $\pi_F$ and $k_F$ be the ring of integers, normalised valuation, choice of uniformiser, and residue field.  Let $E:y^2=x^3+Ax+B$ be a minimal integral model for an elliptic curve defined over $F$ (i.e. $v_F(A)\geq 4\implies v_F(B)<6$), and write $P_E(T):=T^3+A\pi_F^{-2}T+B\pi_F^{-3}$.

In \Cref{tab:reduction at primes >=5} we present the well known summary of the outcome of Tate's algorithm (as presented in \cite{silverman1994advanced}) in this setting.

\begin{table}[ht]
\centering
            \begin{tabular}{|c|c|c|c|}
            \hline
            \textbf{Kodaira Type} & \textbf{Subtype} & $\mathbf{c(E/F)}$ & \textbf{condition}\\
            \hline
            $I_0$ 
                &  
                & $1$ 
                & $v_F(4A^3+27B^2)=0$\\

            \hline
            \multirow{3}{*}{$I_n$}
                & split
                    & $n$ 
                    & $v_F(AB)=0$,\ $v_F(4A^3+27B^2)=n$ and $6B\in k_F^{\times2}$\\
                & nonsplit, $n$ even
                    & $2$
                    & $v_F(AB)=0$,\ $v_F(4A^3+27B^2)=n$ and $6B\not\in k_F^{\times2}$\\
                & nonsplit, $n$ odd
                    & $1$
                    & $v_F(AB)=0$,\ $v_F(4A^3+27B^2)=n$ and $6B\not\in k_F^{\times2}$\\
            \hline
            $II$
                & 
                & $1$
                & $v_F(A)\geq 1$ and $v_F(B)=1$\\
            \hline
            $III$
                & 
                & $2$
                & $v_F(A)=1$ and $v_F(B)\geq 2$\\
            \hline
            \multirow{2}{*}{$IV$}
                & split
                    & $3$
                    & $v_F(A)\geq2$, $v_F(B)=2$ and $B\pi_F^{-2}\in k_F^{\times2}$\\
                & nonsplit
                    & $1$
                    & $v_F(A)\geq2$, $v_F(B)=2$ and $B\pi_F^{-2}\not\in k_F^{\times2}$\\
            \hline
            \multirow{3}{*}{$I_0^*$}
                & nonsplit
                    & $1$
                    & $v_F(A)\geq 2$,\ $v_F(B)\geq 3$,\ $v_F(4A^3+27B^2)=6$,\ and $\#\set{\alpha\in k_F~:~P_E(\alpha)=0}=0$\\
                & partially split
                    & $2$
                    & $v_F(A)\geq 2$,\ $v_F(B)\geq 3$,\ $v_F(4A^3+27B^2)=6$,\ and $\#\set{\alpha\in k_F~:~P_E(\alpha)=0}=1$\\
                & completely split
                    & $4$
                    & $v_F(A)\geq 2$,\ $v_F(B)\geq 3$,\ $v_F(4A^3+27B^2)=6$,\ and $\#\set{\alpha\in k_F~:~P_E(\alpha)=0}=3$\\
            \hline
            \multirow{4}{*}{$I_n^*$}
                & split,\ $n$ even
                    & $4$ 
                    & $v_F(A)=2$,\ $v_F(B)=3$,\ $v_F(4A^3+27B^2)=6+n$ and $-(4A^3+27B^2)\pi_F^{-(6+n)}\in k_F^{\times2}$\\
                & nonsplit,\ $n$ even
                    & $2$
                    & $v_F(A)=2$,\ $v_F(B)=3$,\ $v_F(4A^3+27B^2)=6+n$ and $-(4A^3+27B^2)\pi_F^{-(6+n)}\not\in k_F^{\times2}$\\
                & split,\ $n$ odd
                    & $4$
                    & $v_F(A)=2$,\ $v_F(B)=3$,\ $v_F(4A^3+27B^2)=6+n$ and $6B(4A^3+27B^2)\pi_F^{-(9+n)}\in k_F^{\times2}$\\
                & nonsplit,\ $n$ odd
                    & $2$
                    & $v_F(A)=2$,\ $v_F(B)=3$,\ $v_F(4A^3+27B^2)=6+n$ and $6B(4A^3+27B^2)\pi_F^{-(9+n)}\not\in k_F^{\times2}$\\
            \hline
            \multirow{2}{*}{$IV^*$}
                & split
                    & $3$
                    & $v_F(A)\geq3$,\ $v_F(B)= 4$\ and $B\pi_F^{-4}\in k_F^{\times2}$\\
                & nonsplit
                    & $1$
                    & $v_F(A)\geq3$,\ $v_F(B)=4$\ and $B\pi_F^{-4}\not\in k_F^{\times2}$\\
            \hline
            $III^*$
                & 
                & $2$
                & $v_F(A)=3$ and $v_F(B)\geq 5$\\
            \hline
            $II^*$
                & 
                & $1$
                & $v_F(A)\geq 4$ and $v_F(B)=5$\\
            \hline
            \end{tabular}
            \caption{Tate's Algorithm for a minimal model in residue characteristic at least $5$}\label{tab:reduction at primes >=5}
        \end{table}
    \end{landscape}

\end{document}